\documentclass[reqno]{amsart}
\usepackage{amsfonts}
\usepackage{amsthm}
\usepackage{amssymb}
\usepackage{dsfont}
\usepackage{amscd}
\usepackage{color}

\newtheorem{theorem}{Theorem}[section]
\newtheorem{corollary}[theorem]{Corollary}

\newtheorem{prop}[theorem]{Proposition}
\theoremstyle{definition}

\theoremstyle{remark}
\newtheorem{remark}{Remark}[section]

\newtheorem*{acknow}{Acknowledgments}
\numberwithin{equation}{section}
\newcommand{\abs}[1]{\left\vert#1\right\vert}

\def\be{\begin{equation}}
\def\ee{\end{equation}}
\def\ba{\begin{eqnarray*}}
\def\ea{\end{eqnarray*}}
\def\bae{\begin{eqnarray}}
\def\eae{\end{eqnarray}}
\def\bc{\begin{center}}
\def\ec{\end{center}}

\begin{document}

\title[Products of complex Ginibre matrices with a source]{Singular values for products of complex Ginibre matrices with a source: hard edge limit and phase transition}


\author{Peter J. Forrester} \address{Department of Mathematics and Statistics, The University of Melbourne, Victoria 3010, Australia;
 ARC Centre of Excellence for Mathematical \& Statistical Frontiers}\email{p.forrester@ms.unimelb.edu.au}
\author{Dang-Zheng Liu} \address{Key Laboratory of Wu Wen-Tsun Mathematics, Chinese Academy of Sciences, School of Mathematical Sciences, University of Science and Technology of China, Hefei 230026, P.R.~China}
\email{dzliu@ustc.edu.cn}

\date{\today}


\begin{abstract} The singular values squared of the random matrix product
$Y = G_r G_{r-1} \cdots G_1 (G_0 + A)$,
where each $G_j$ is a rectangular  standard complex Gaussian matrix  while $A$ is
non-random, are shown to be a determinantal point process with correlation kernel given by a double contour integral. When all but finitely many eigenvalues of $A^*A$ are   equal to $bN$, the
kernel is shown to admit a well-defined hard edge scaling, in which  case a critical value is established  and
a phase transition phenomenon is observed. More specifically, the limiting   kernel in the subcritical regime of $0<b<1$ is independent of $b$, and is in fact the same
as that known for the case $b=0$ due to Kuijlaars   and Zhang. The critical regime of $b=1$   allows for a double scaling limit  by choosing $b = (1 -  \tau/\sqrt{N})^{-1}$, and for this the critical kernel and  outlier phenomenon  are established.
  In the simplest case $r=0$, which is closely related to
 non-intersecting squared Bessel paths, a distribution corresponding to the finite shifted mean  LUE   is proven to be the scaling limit in the supercritical regime of $b>1$ with two distinct scaling rates. Similar results also hold true for the random matrix  product  $T_r T_{r-1} \cdots T_1 (G_0 + A)$, with each $T_j$ being a truncated unitary matrix.
\end{abstract}

\maketitle
\section{Introduction and main results} \label{sectionintroduction}

\subsection{Introduction}
The squared singular values of a matrix $X$ are equal to the eigenvalues of the positive semi-definite Hermitian matrix $X^* X$, where $X^*$ denotes the
Hermitian conjugate of $X$. An ensemble of random matrices of the form $X^* X$ may then contain $x=0$ as the left boundary of support of the eigenvalues.
Since the eigenvalue density is strictly zero for $x<0$, $x=0$ is then called a hard edge (see e.g.~\cite[Ch.~7]{Fo10}).
As an explicit example, consider the ensemble of $n \times N$ ($n \ge N$) rectangular standard complex Gaussian random matrices, namely the joint density of elements being proportional to $\exp\{-\textrm{tr}(X^* X)\}$,  and let $X$ be a matrix
from this ensemble. Let
$\{\lambda_j\}$ denote the eigenvalues of the scaled positive semi-definite matrix $N^{-1}X^* X$. In the limit $N \to \infty$ with $n - N$ fixed, the density of
$\{\lambda_j\}$ has support $[0,4]$. That the support is a finite interval gives rise to this particular scaling being referred to as global scaling, and the corresponding
density as the global density.
The explicit functional form of the global density is given by the so-called Marchenko-Pastur law (see e.g.~\cite{PS11})
\begin{equation}\label{2.1}
\rho_{(1)}^{\rm MP}(\lambda) = {1 \over 2 \pi} \sqrt{  4 - \lambda \over\lambda}, \qquad 0 < \lambda \le 4.
\end{equation}
Note in particular the reciprocal square root singularity as the hard edge $\lambda = 0$ is approached from above, in contrast to the square root
singularity as $\lambda \to 4^-$. The point $\lambda = 4$ is an example of what is termed a soft edge, since for finite $N$ the eigenvalue density is
not strictly zero for $\lambda > 4$.

Continuing with this example, for large $N$  the eigenvalues in the neighbourhood of the hard edge have spacing $\mathcal{O}(1)$ upon the
introduction of the scaled variables $X_j = 4 N^{2} \lambda_j$ $(j=1,\dots,N)$ (see e.g.~\cite[\S 7.2.1]{Fo10}).
This will be referred to as hard edge scaling. Moreover, in the limit $N \to \infty$,
and with $\nu_0 = n - N$, the limiting state  --- referred to as the hard edge state --- is an example of a determinantal point process, meaning that the
$k$-point correlation function can be written in the form
\begin{equation}\label{3.1}
\rho_{(k)}(X_1,\dots,X_k) = \det [ K^{\rm h}(X_j, X_l) ]_{j,l=1,\dots,k}
\end{equation}
with correlation kernel (see e.g.~\cite[Exercises 7.2 q.1]{Fo10})
\begin{equation}\label{3.2}
 K^{\rm h}(x,y)  = {1 \over 4}  \int_0^1 J_{\nu_0}(\sqrt{xt}) J_{\nu_0}(\sqrt{yt}) \, dt,
 \end{equation}
where $J_{\nu_0}(x) $ is the  Bessel function of the first kind of order $\nu_0$.

 Our interest in this paper is in the functional form and analytic properties of the correlation kernel for the hard edge scaling of the squared singular values of the product
 of independent random matrices
 \begin{equation}\label{Y}
 Y = G_r G_{r-1} \cdots G_1 (G_0 + A),
 \end{equation}
where each $G_j$ is an $(N+\nu_j)\times (N+\nu_{j-1})$    standard complex Gaussian matrix (also referred to as the complex Ginibre matrices since such  non-Hermitian
random matrices in the  square  case were  first studied by Ginibre \cite{ginibre56}) with $\nu_{-1}=0$ and integers $\nu_0, \ldots, \nu_r\geq 0$, while $A$ is of size $(N + \nu_0) \times N$ and
fixed.  If we focus on the singular values of $Y$, the definition (\ref{Y}) can equivalently be written as a product of independent square matrices $G_j$ now with each
being distributed according to the joint density of elements proportional to $\det^{\nu_j}(G_j^* G_j)\exp\{-\textrm{tr}(G_j^* G_j)\}$; see e.g. \cite{AIK13,KS14}. In this case the restriction on the parameters can be relaxed to $\nu_0, \nu_1, \ldots \nu_r>-1$, and the main results  in the present paper (for instance, Proposition \ref{pdf}, Theorems \ref{criticalkernel}, \ref{deformedcriticalkernel},  \ref{noncriticalkernel} and \ref{supercriticalkernel}) also hold true.

In the case that all entries of $A$ are zero, the determinantal representation of the joint eigenvalue density and  the limiting hard edge state have been the subject of a number of recent works
\cite{AIK13,AKW13,Fo14,KS14,KZ,st2014}.
For $r$ nonzero, the product (\ref{Y}) is the simplest nontrivial example of the
more general product $(G_r+A_{r}) \cdots (G_1+A_1) (G_0 + A_0)$, where each $G_j$ is random and each $A_j$ is fixed.
However
for the latter  product with $r\geq 1$, it is not known how to find a closed form of  the joint eigenvalue density,
which is the starting point of our study (see Proposition \ref{pdf} below), let alone to study asymptotic  statistical  properties.

The study of products of random matrices goes back to the pioneering work of Furstenberg and Kesten \cite{FK60} in the context of dynamical systems and their Lyapunov exponents.  Later, applications were found in Schr\"{o}dinger operator theory \cite{BL85}, in statistical physics relating to disordered and chaotic dynamical systems \cite{CPV93}, in wireless communication  networks \cite{TV04} and in combinatorics \cite{pz}. Further motivations from theoretical physics include  the DMPK equation
in mesoscopic quantum transport \cite{Be97,Do82,MPK88}, fluid turbulence \cite[eqn.(16)]{FGV01}  and time evolution models \cite[eqn.(1)]{May72} (the last two models can be immediately
recognized after discrete integration in time).

Products of complex Ginibre matrices, and of truncations of Haar distributed random unitary matrices, have attracted much
attention as examples of determinatal point processes with kernels possessing special integrability properties. The first advance in this direction was by
Akemann and coworkers \cite{AKW13,AIK13}, who derived the joint eigenvalue density and corresponding correlation
kernel in terms of orthogonal functions for the product  of complex Ginibre matrices; a double integral formula for the
correlation kernel was subsequently obtained by Kuijlaars and Zhang, cf. \cite[Prop.~5.1]{KZ}. These advances have opened up the possibility to study local statistical properties, see \cite{KZ,Fo14,KS14,KKS15} for the   hard edge limit and \cite{LWZ14} for bulk and soft edge limits. All these studies  form part of a fast paced and very recent literature relating to the integrability and exactly solvable properties of random matrix products.
Works relating to this theme which have appeared on the electronic preprint archive over the past few months (as of July 2015) include
\cite{LWZ14,GNT14,KKS15,Ku15,Ki15,Zh15,FW15,CKW15}; we refer the reader to \cite{AI15} for a recent survey article.
Here we contribute to this line of research by undertaking a comprehensive study of the hard edge state
formed by the singular values of (\ref{Y}).

The first point to note is that
the hard edge state in the case $A=0$ depends on $r$, and thus is no longer described by the correlation kernel (\ref{3.2}).
This fact  can be anticipated by an analysis
of the global density of the squared singular values \cite{agt3,FL14,neu,pz}. The global density,  which refers to the limiting density of eigenvalues of $N^{-r-1}Y^{*}Y$ as $N\rightarrow \infty$, is found to exhibit the hard edge singularity (see \cite{pz} or  \cite[eqn~(2.16)]{FL14})
\begin{equation}\label{f1}
 \frac{1}{\pi}\sin \frac{\pi}{r+2}\,   \lambda^{-1 + \frac{1}{r+2}} \qquad {\rm as} \quad \lambda \to 0^+,
\end{equation}
which has an $r$-dependent exponent. In fact with the eigenvalues of $Y^* Y$ scaled according to $X_j = N x_j$ $(j=1,\dots,N)$, as $N\rightarrow \infty$ the hard edge state in the case $A=0$
forms a determinantal point process with limiting correlation kernel
\begin{align}\label{r}
K^{{\rm h}, r} (x,y) & =
 {1 \over (2 \pi i)^2} \int_{-1/2 - i \infty}^{-1/2 + i \infty} du  \oint_{\Sigma} dt \,
   \prod_{j=-1}^r { \Gamma(\nu_j + u+1 ) \over   \Gamma(\nu_j + t+1 ) }
{\sin \pi u \over \sin \pi t}      {x^t y^{-u-1} \over u - t} \nonumber \\
& =
 \int_0^1   G^{1,0}_{0,r+2} \Big ({ \underline{\hspace{0.5cm}}
 \atop 0,-\nu_0,-\nu_1,\dots,-\nu_r} \Big | u x \Big ) G^{r+1,0}_{0,r+2} \Big ({ \underline{\hspace{0.5cm}}
 \atop \nu_0, \nu_1,\dots,\nu_r,0} \Big | u y \Big ) \, du,
\end{align}
where $G^{m,n}_{p,q}$ denotes the Meijer G-function defined by the contour integral
\begin{align}  G^{m,n}_{p,q} \Big ({a_1,\dots,a_p
 \atop b_1,\dots,b_q} \Big |z \Big )=  \frac{1}{2\pi i} \int_{\gamma}    \frac{\prod_{j=1}^{m} \Gamma(b_j+s)\prod_{j=1}^{n} \Gamma(1-a_j-s)}{\prod_{j=m+1}^{q} \Gamma(1-b_j-s)\prod_{j=n+1}^{p} \Gamma(a_j+s)}
  z^{-s} ds, \label{Gfunction}\end{align}
see \cite[Sect.~5.2]{Luke}   for the choice of the contour $\gamma$ and elementary properties of G-functions, or \cite{BS13} for a gentle introduction; it is worth mentioning a particular relation  between the generalized hypergeometric function ${}_pF_{q} $ and the Meijer G-function (cf. eqn~(14), \cite[Sect.~5.2]{Luke})
 \begin{multline} \, {}_p F_{q}\big(a_1,  \dots,a_p; b_1,  \dots,b_q;z\big) =  \\ \frac{\prod_{l=1}^{q} \Gamma(b_l+1) }{\prod_{l=1}^{p} \Gamma(a_l+1) } G^{1,p}_{p,q+1} \Big ({  1-a_1, \dots,1-a_p
 \atop 0,1-b_1, \dots,1-b_q} \Big |  -z \Big ),\label{FGrelation}\end{multline}
 which thus gives an alternative way of writing the first Meijer G-function in the final line of (\ref{r}).
These kernels were   described  in    \cite{KZ} and   are  named after    Meijer G-kernels in \cite{KS14}.  They also appear in the hard edge scaling for products with inverses of Ginibre matrices \cite{Fo14},  products of truncated unitary matrices \cite{KS14},   Cauchy two matrix models \cite{BB14,BGS14,FK14}, and Muttalib-Borodin biorthogonal ensembles \cite{Bo98,Mu95} (cf.~\cite{KS14} for the relationship between Borodin's expression and   Meijer G-kernels).

  As noted in \cite[Sect.~5.3]{KZ}, in the case $r=0$ the facts that
\begin{equation}\label{GJ}
G^{1,0}_{0,2} \Big ({ \underline{\hspace{0.5cm}}
 \atop 0,-\nu} \Big | u x \Big ) = (ux)^{-\nu/2} J_\nu(2 \sqrt{ux}), \ \
 G^{1,0}_{0,2} \Big ({ \underline{\hspace{0.5cm}}
 \atop \nu,0} \Big | u y \Big ) = (uy)^{\nu/2} J_\nu(2 \sqrt{uy}),
 \end{equation}
 show
 \begin{equation}\label{GJa}
 K^{{\rm h},0}(x,y)= 4 ( {y / x})^{\nu/2} K^{{\rm h}}(4x,4y).
 \end{equation}
 The factor of $ ( y / x  )^{\nu/2} $ cancels out of the determinant (\ref{3.1}), while the factors of $4$ are accounted
 for by this same factor being present in the scaling leading to (\ref{3.2}); recall the text leading to this equation.

Consider now (\ref{Y}) with
\begin{equation}\label{AN}
A = \sqrt{b N} I_{(N + \nu_0)  \times N},
\end{equation}
where $I_{(N + \nu_0) \times N}$ denotes the $(N + \nu_0) \times N$ rectangular matrix with $1$'s on the main diagonal,  and 0's
elsewhere. It was shown recently in \cite[Remark 3.4]{FL14} that there is a critical value of $b=1$   for which as $N \to \infty$ the left hand edge of the support of the
global scaled squared singular values equals 0 for the last time as $b$ increases from 0. Moreover, it was shown that the singularity of the global density has the
leading form
\begin{equation}\label{f2}
\frac{1}{\pi}\sin \frac{2\pi}{2r+3}\,   \lambda^{-1 + \frac{2}{2r+3}} \qquad {\rm as} \quad \lambda \to 0^+,
\end{equation}
which gives rise to a different family of exponents to those in (\ref{f1}). We remark that the fractional part of the exponents,  $1/(r+2)$ and $1/(r+3/2)$ respectively in  \eqref{f1} and \eqref{f2},  are the reciprocals of positive integers and half-integers, which given knowledge of the correlation kernel (\ref{r}) and its analogue in relation to (\ref{f2}) to be established herein (see eqn.~(\ref{criticalk}) below),
is coincident with them being the simplest in terms of tractability of the general rational fractional exponents accessible in the
Raney family (see e.g.~\cite[eqn.~(2.16)]{FL14}), so named due to the sequence formed by the moments of the global density.

Let $A$ be again given by (\ref{AN}),   and  consider the case $r=0$ in  \eqref{Y}  so that $Y = G_0 + \sqrt{b N} I_{(N+\nu_0) \times N}$.
It is well known that the squared singular values allow for an interpretation as the positions of non-intersecting particles on the half line evolving
according to the squared Bessel process with parameter $d=2(\nu_0 +1)$ (see e.g.~\cite{KO01,KT10}).
In this interpretation the particles all begin at the same point ${bN}$, evolve for $t=1$ time units, and furthermore are conditioned to remain
non-intersecting if the process was to continue to $t \to \infty$.
The support of the density of such a process with a delta function
initial condition $N \delta(x-a)$ is, for $Nt < a$ equal at leading order to $[- Nt + a,  Nt + a]$ (this fact is implied by results in \cite{KMW09}, for example),
so we see that with $a =  bN$ and $t=1$, the particles first come in contact with the wall $x=0$ as $b$ is decreased to $b=1$.
A functional form of the hard edge scaled kernel
in the critical case $b=1$, generalised to a double scaling by setting $b = (1 -  \tau/\sqrt{N})^{-1}$, has recently been obtained in \cite{KFW11}.
In the present paper an alternative functional form to that in \cite{KFW11} is derived; see eqn.~(\ref{3.24}) below.
The kernel (\ref{3.24}), further specialised to $\nu_0 = -1/2$ reads
\begin{equation}\label{N1}
{1 \over 2 \pi^2 i} \Big ( {1 \over \xi \eta} \Big )^{1/4}
\int_0^\infty du \int_{i\mathbb{R}} dv \, {e^{-\tau u + {1 \over 2} u^2 + \tau v + {1 \over 2} v^2} \over u - v}
{\cos (2 \sqrt{u\xi}) \over u^{1/2} }\cos (2 \sqrt{v \eta}).
\end{equation}
With $\xi$ and $\eta$ replaced by squared variables, (\ref{N1}) is identified in \cite{KFW11} as the symmetric Pearcey kernel found in the study
\cite{BK10}.
Moreover, our method of derivation of this new functional form  in the case $r=0$ works equally as well for the double scaling of the critical kernel in the general $r$ case, which is our main theme.
The resulting explicit double contour integral expression is given in Theorem \ref{criticalkernel}  below.

\subsection{Main results} In preparation for the statement of our first  key result,
 let us introduce two auxiliary  functions. The first is defined to be
  \begin{multline}\Psi(u;x)= \frac{1}{(2\pi i)^{r}}    {1 \over \Gamma(\nu_0+1)} \int_{\gamma_1}d w_1\cdots \int_{\gamma_r}d w_r  \prod_{l=1}^r w_{l}^{-\nu_l-1}e^{w_l}\\ \times   e^{x/(w_1\cdots w_r)}\,
   {}_0F_1\big(\nu_0+1;
  - ux/(w_1\cdots w_r)\big),\label{function3}\end{multline}
where  $\gamma_1, \ldots, \gamma_r$ are paths  starting and ending at negative infinity and encircling the origin once in the positive direction,
or equivalently (cf. \eqref{function4.1-1}, Sect. \ref{sectiontruncatedunitary} below),
\begin{multline}\Psi(u;x)= {1 \over \Gamma(\nu_0+1)}  \frac{1}{ 2\pi i }       \int_{\gamma}d w \,(-x)^{-w} {}_1F_1\big(w;\nu_0+1;u\big) \\ \times \Gamma(w)
   \prod_{l=1}^{r} \frac{1}{\Gamma(\nu_l+1-w)}
   ,\label{function3-1}\end{multline}
 with $\gamma$    encircling all non-positive integers,
 while the other reads
   \be \Phi(v;y)= \frac{1}{2\pi i} \int_{c-i\infty}^{c+i\infty} ds\, y^{-s}  \phi(v;s) \prod_{l=1}^{r} \Gamma(\nu_l+s),\label{function4} \ee
 where $c>-\min\{\nu_0, \nu_1, \ldots, \nu_r\}$, and \begin{align} \phi(v;s)&=\frac{1}{\Gamma(\nu_0+1)}\int_{0}^{\infty} dt \, t^{\nu_0+s-1}e^{-t}{}_0F_1(\nu_0+1;-vt) \label{function2}\\
 &=\frac{\Gamma(\nu_0+s)}{\Gamma(\nu_0+1)}\, {}_1F_1(\nu_0+s;\nu_0+1;-v).  \label{function1}\end{align}
 In the case $r=0$ (\ref{function3}) is to be interpreted as
 \be \label{E0a}
\Psi(u;x)=    \frac{1}{\Gamma(\nu_0+1)} \, e^{x}\, {}_0F_1(\nu_0+1;-ux),
\ee
and a calculation shows that  (\ref{function4}) simplifies to read
\be \label{E1a}
\Phi(v;y)=   \frac{1}{\Gamma(\nu_0+1)} \, y^{\nu_0} e^{-y}\, {}_0F_1(\nu_0+1;-vy).
\ee
The two auxiliary functions appear in a double contour integral expression for the correlation kernel, which we
present next. Its significance is that it provides the starting point for further asymptotic analysis. The special case $r=0$ was previously obtained by Desrosiers and one of the present authors; see ~\cite[Prop.~5]{DF08}.

 \begin{prop}\label{pdf} Let   $Y$ be  defined by   (\ref{Y}), and suppose that all   eigenvalues  $a_1, \ldots, a_N$  of  $A^*A$  are positive.  The joint density  of  eigenvalues  for $Y^*Y$ can be
 written in the form
 \be \mathcal{P}_{N}(x_1,\ldots,x_N)=\frac{1}{N!} \det[K_N(x_i,x_j)]_{i,j=1}^{N} \label{deterpdf}
 \ee
with correlation kernel
 \be
K_N(x,y)=\frac{1}{2\pi i}\int_{0}^{\infty} du \int_{\mathcal{C}} dv \, u^{\nu_0}e^{-u+v} \Psi(u;x) \Phi(v;y)\frac{1}{u-v}\prod_{l=1}^{N}\frac{u+a_l}{v+a_l}, \label{kernelCD}\ee
where $\mathcal{C}$ is a counterclockwise contour encircling $-a_1,\ldots, -a_N$ but not $u$.
 \end{prop}

\begin{remark} \label{nullremark}
When some of the   parameters $a_l$'s are null, the double integral representation  \eqref{kernelCD} remains valid provided that  $\int_{0}^{\infty} du$  is interpreted as $\lim_{\varepsilon\rightarrow 0^+}\int_{\varepsilon}^{\infty} du$, or   for given $u>0$ $\mathcal{C}$ is chosen such that $\textrm{Re} \{v\}<u$ with any  $v \in \mathcal{C}$. \end{remark}

One of the  main results in the present paper concerns a  double scaling limit near the  critical point, which permits a new family of  limiting kernels.
\begin{theorem}[Critical kernel]\label{criticalkernel}  With the  kernel \eqref{kernelCD},   for $\tau \in \mathbb{R}$ let
\be a_1= \cdots= a_N= N (1- \tau/\sqrt{N} )^{-1}.
 \label{doublescaling}\ee     Then we have
\begin{align}  \lim_{N\rightarrow \infty}\frac{1}{ \sqrt{N}} K_N\Big(\frac{\xi}{\sqrt{N}},\frac{\eta}{\sqrt{N}}\Big) & =   \frac{1}{2\pi i}\int_{0}^{\infty}du \int_{ i\mathbb{R}} dv \ \Big (\frac{u}{v} \Big )^{\nu_0}\frac{e^{- \tau u-\frac{1}{2}u^2 +\tau v+\frac{1}{2}v^2}}{u-v} \nonumber \\  \times \,
G_{0,r+2}^{1,0}& \Big({\atop 0, -\nu_0, -\nu_1,\ldots,-\nu_r}\Big|u\xi\Big) G_{0,r+2}^{r+1,0} \Big({\atop \nu_0,  \nu_1,\ldots,\nu_r,0}\Big|v\eta\Big) \nonumber \\
& =: {\mathcal K}^{{\rm h},r}(\xi,\eta;\tau), \label{criticalk}\end{align}
valid uniformly for $\xi, \eta$ in any compact set  of $(0,\infty)$ and for $\tau$ in any compact set  of $\mathbb{R}$.
 \end{theorem}

In the special case $r=0$, upon making use of (\ref{GJ}) we see from (\ref{criticalk}) that
\begin{multline}\label{3.24}
   \Big ( {\xi \over \eta} \Big )^{\nu_0/2}   \mathcal K^{{\rm h},0}\Big ( {\xi  }, {\eta  };\tau \Big )  =
 \frac{1}{2\pi i}\int_{0}^{\infty}du \int_{i\mathbb{R}} dv \ \Big (\frac{u}{v} \Big )^{\nu_0/2}\frac{e^{- \tau u-\frac{1}{2}u^2 +\tau v+\frac{1}{2}v^2}}{u-v}  \\  \times \,
  J_{\nu_0}(2\sqrt{u\xi}) J_{\nu_0}(2\sqrt{v\eta}),
\end{multline}
where the integral form on the RHS of the above equation  is similar to \eqref{kernelCD} with $r=0$ (cf.~\cite[Prop.~5]{DF08}).
In the study \cite[displayed equation below (1.34)]{KFW11}, this kernel was conjectured to be an equivalent functional form to that derived therein in the  case $r=0$. Our work thus provides a direct way of deriving (\ref{3.24}) for the $r=0$ critical kernel.
 Recently, an understanding of the resulting functional identity has been given in   \cite[Remark 2.26]{DV14}.

Theorem \ref{criticalkernel} quantifies the limiting correlation kernel for the situation that $a_k = N(1- \tau/\sqrt{N})^{-1}$ $(k=1, \ldots, N)$, which is shown to depend on $\tau$, thus
justifying the term critical kernel. A variation on this setting is to have at most finitely many source eigenvalues, say $a_1, \ldots, a_m$, go to infinity   at a   smaller but appropriate scale and others  remain  at the same critical value. This gives rise to  a multi-parameter     deformation of  the critical kernel   \eqref{criticalk}.

\begin{theorem} [Deformed critical kernel]\label{deformedcriticalkernel}
 With   the kernel    \eqref{kernelCD}, for a fixed nonnegative  integer $m$,  let
\be a_j=\sqrt{N}\sigma_j, \,  j=1, \ldots, m \  \mbox{and} \  a_k=N(1- \tau/\sqrt{N})^{-1}, \, k=m+1, \ldots, N,
 \label{mconfluentsource}\ee
 where $\tau\in \mathbb{R}$ and $\sigma_1, \ldots, \sigma_m>0$. Then   we have
\begin{align}  \lim_{N\rightarrow \infty}\frac{1}{ \sqrt{N}} K_N\Big(\frac{\xi}{\sqrt{N}},\frac{\eta}{\sqrt{N}}\Big) & =   \frac{1}{2\pi i}\int_{0}^{\infty}du \int_{ -c-i\infty}^{-c+i\infty} dv \ \Big (\frac{u}{v} \Big )^{\nu_0}\frac{e^{- \tau u-\frac{1}{2}u^2 +\tau v+\frac{1}{2}v^2}}{u-v} \nonumber \\  \times \, \prod_{j=1}^{m}\frac{u+\sigma_j}{v+\sigma_j}\,
G_{0,r+2}^{1,0}& \Big({\atop 0, -\nu_0, -\nu_1,\ldots,-\nu_r}\Big|u\xi\Big) G_{0,r+2}^{r+1,0} \Big({\atop \nu_0,  \nu_1,\ldots,\nu_r,0}\Big|v\eta\Big) \nonumber \\
& =: {\mathcal K}^{{\rm h},r}_{m}(\xi,\eta;\tau,{\sigma}), \label{deformedcriticalk}\end{align}
where $0<c<\min\{\sigma_1,\ldots,\sigma_m\}$.
 \end{theorem}

In the simplest case $r=0$, upon making use of (\ref{GJ}) we see from (\ref{deformedcriticalk}) that
\begin{multline}\label{r=0m}
   \Big ( {\xi \over \eta} \Big )^{\nu_0/2}   \mathcal K^{{\rm h},0}_{m}\Big ( {\xi  }, {\eta  };\tau \Big )  =
 \frac{1}{2\pi i}\int_{0}^{\infty}du \int_{ -c-i\infty}^{-c+i\infty} dv   \\  \times \Big (\frac{u}{v} \Big )^{\nu_0/2}\frac{e^{- \tau u-\frac{1}{2}u^2 +\tau v+\frac{1}{2}v^2}}{u-v}  \, \prod_{j=1}^{m}\frac{u+\sigma_j}{v+\sigma_j}\,
  J_{\nu_0}(2\sqrt{u\xi}) J_{\nu_0}(2\sqrt{v\eta}). \
\end{multline}
Even in this special case, the kernel (\ref{deformedcriticalk}) appears to be  new.

We remark that the inter-relationship  between the interpolating kernel \eqref{deformedcriticalk} and critical kernel \eqref{criticalk} is similar in form to that  between the interpolating Airy kernel  and Airy kernel (see e.g.~\cite{BBP,ADvM09}). Furthermore,  as the parameter  $b$ displayed in eqn.~(\ref{AN}) increases from zero,  we will  establish a phase transition at the hard edge from the Meijer G-kernel  (cf.~Theorem  \ref{noncriticalkernel})  to the critical  and deformed critical kernels (cf.~Theorems   \ref{criticalkernel} and \ref{deformedcriticalkernel}), then to the shifted mean LUE   kernel   (cf.~Theorem \ref{supercriticalkernel}); see Section \ref{sectionhardlimit} for more details. A similar phase transition occurs in another  random matrix  product
 $T_r T_{r-1} \cdots T_1 (G_0 + A)$, with each $T_j$ being a truncated unitary matrix; see Section \ref{sectiontruncatedunitary}.

The paper is organized as follows.  Section \ref{sectioneigenvaluepdf} is devoted to the joint eigenvalue   probability density function (PDF)  and a double contour integral representation for the correlation kernel of the squared singular values of the   product \eqref{Y}.  The proof of Proposition \ref{pdf} will be given,  and
the formulas for the average  of the  ratio of characteristic polynomials and a single (inverse) characteristic polynomial are also derived. In Section \ref{sectionhardlimit} the hard edge limits of the kernel in different regimes are evaluated, which include the proofs of Theorems \ref{criticalkernel} and \ref{deformedcriticalkernel}. Our methods are used to similarly analyse the product of $r$ truncated unitary matrices and one shifted mean Ginibre matrix
in Section \ref{sectiontruncatedunitary}. In Section \ref{sectiondiscussion} further discussions on asymptotics for large   variables, and some open problems, are presented.

\section{Eigenvalue PDF and double integral  for correlation kernel} \label{sectioneigenvaluepdf}

\subsection{Correlation kernels}
Consider (\ref{Y}) in the case $r=0$. Let $x_1,\dots,x_N$ and $a_1,\dots,a_N$ denote the eigenvalues of $Y^*Y$ and $A^*A$ respectively. It is well known
(see e.g.~\cite[Prop.~5]{DF08}, \cite[\S 11.6]{Fo10}) that the eigenvalue  PDF of the random matrix $Y^*Y$ is an example of a biorthogonal ensemble
\cite{Bo98}
 \be \mathcal{Q}_{N}(x_1,\ldots,x_N)=\frac{1}{Z_N} \det[\eta_i(x_j)]_{i,j=1}^{N}\det[\xi_i(x_j)]_{i,j=1}^{N},\label{chGUEsourcepdf}\ee
where  $\eta_i(x)=x^{i-1}$,  $\xi_i(x)=x^{\nu_0}e^{-x} {}_0F_1(\nu_0+1;a_i x)$, and $Z_N$ denotes the normalisation. Our first task is to specify a functional form
for the joint eigenvalue PDF of $Y^* Y$ in the case of general $r$. For this purpose use will be made of a recent result due to Kuijlaars and Stivigny \cite{KS14}.

\begin{prop}[Special case of {\cite[Thm.~2.1]{KS14}}] \label{P5}
Let $W$ be an $n \times n$ random matrix, and suppose that  the eigenvalue PDF of $W^* W$  can be written in the form
\begin{equation}\label{2.3w}
\prod_{1 \le j < k \le n} (x_k - x_j) \det [ f_{k-1}(x_j) ]_{j,k=1}^n
\end{equation}
for some $\{f_{k-1}(x) \}_{k=1,\dots,n}$.  For $\nu \ge 0$, let $G$ be an $(n + \nu) \times n$  standard complex Gaussian matrix. The squared singular values of $GW$, or equivalently the eigenvalues of $(GW)^{*} GW$, then have their PDF proportional to
\begin{equation}\label{2.3wA}
\prod_{1 \le j < k \le n} (y_k - y_j)  \det [ g_{k-1}(y_j) ]_{j,k=1}^n,
\end{equation}
where
\begin{equation}\label{2.3wB}
g_k(y) = \int_0^\infty x^\nu e^{-x} f_k \Big ( {y \over x} \Big ) \, {dx \over x}, \qquad
k=0,\dots,n-1.
\end{equation}
\end{prop}

\medskip
Let $Y$ be defined in \eqref{Y} and let  $a_1,\dots,a_N$ denote the eigenvalues of   $A^*A$.   Starting with (\ref{chGUEsourcepdf}), application of Proposition \ref{P5} $r$ times in succession shows that the joint eigenvalue PDF of $Y^* Y$  is equal to
 \begin{equation}   \frac{1}{Z_N}\det[\eta_i(x_j)]_{i,j=1}^{N}\det[\xi_i(x_j)]_{i,j=1}^{N}, \label{B}\ee
where $\eta_i(x)=x^{i-1}$ and $\xi_j(x)=\Phi(-a_j;x)$, while with $T=t_1 \cdots t_r$
 \be \label{T} \Phi(v;y)=\frac{1}{\Gamma(\nu_0+1)}\int_{0}^{\infty} d t_1 \cdots \int_{0}^{\infty} d t_r  \prod_{i=1}^{r}  t_{i}^{\nu_i-1}e^{-t_{i}}\,  (\frac{y}{T})^{\nu_0} e^{-\frac{y}{T}}   {}_0F_1(\nu_0+1;-v\frac{y}{T}), \end{equation}
 valid for $r \ge 1$ (for $r=0$ $\xi_j(x)$ is defined as below (\ref{chGUEsourcepdf})). Here $\Phi(v;y)$ is actually the same  as defined in  \eqref{function4}, for which   application of the Mellin transform
gives
  \be \int_{0}^{\infty} y^{s-1}\Phi(v;y) \, dy =\phi(v;s) \prod_{l=1}^{r} \Gamma(\nu_l+s),\label{mellin1}\ee
while   use of the inverse Mellin transform gives the sought expression.
We stress that when some of the $a_j$'s in \eqref{B} coincide L'Hospital's rule provides the appropriate eigenvalue density.

The significance of the structure (\ref{B}) is that it provides a systematic way to compute the corresponding $k$-point correlation function  $\rho_{(k)}(x_1,\dots,x_k)$, where the normalization has been chosen such that
integrating gives  $N(N-1)\cdots (N-k+1)$, see e.g. \cite[eqn (5.1)]{Fo10} for the definition.

\begin{prop}[{\cite[Prop.~2.2]{Bo98}}] \label{PB}
With $g_{i,j} :=  \int_0^\infty \eta_i(x) \xi_j(x) \, dx$, let $[g_{i,j}]_{i,j=1}^{n}$ be invertible for each
$n=1,2,\dots$. Defining $c_{i,j}$ by
\begin{equation}\label{7.eb1'}
\big([c_{i,j}]_{i,j=1}^{N}\big)^t =
\big ( [g_{i,j}]_{i,j=1}^{N} \big )^{-1},
\end{equation}
and setting
\begin{equation}\label{7.eb2'}
K_N(x,y) =  \sum_{i,j=1}^N c_{i,j} \eta_i(x) \xi_j(y),
\end{equation}
we have that the $k$-point correlation function is given by
\begin{equation}\label{7.eb2}
\rho_{(k)}(x_1,\dots,x_k) =
\det [ K_N(x_j,x_l) ]_{j,l=1}^{k}.
\end{equation}
\end{prop}

We are now ready to complete the proof of Proposition \ref{pdf}.
\begin{proof}[Proof of Proposition \ref{pdf}] Our first task is to compute $g_{i,j} :=  \int_0^\infty \eta_i(x) \xi_j(x) \, dx$. For this purpose,
we require the fact (see e.g.~\cite[eqns.~(6.2.15), (6.2.33) and (4.5.2)]{AAR} ) that with
 $L_{n}^{\nu_0}(y)$ denoting the Laguerre polynomial of degree $n$,  one has the Hankel pair
\be L_{n}^{\nu_0}(y)=\frac{e^{y}}{n!\Gamma(\nu_0+1)} \int_{0}^{\infty} t^{\nu_0+n}e^{-t}{}_0F_1(\nu_0+1;-yt) \, dt \label{laguerreformula1}\ee
and \be t^n= \frac{n!e^{t}}{\Gamma(\nu_0+1)} \int_{0}^{\infty}  y^{\nu_0}L_{n}^{\nu_0}(y)e^{-y}{}_0F_1(\nu_0+1;-ty)\, dy. \label{laguerreformula2}\ee
Combination of \eqref{mellin1}, \eqref{function2} and \eqref{laguerreformula1} shows that
\be g_{i,j}=(i-1)!e^{a_j}L_{i-1}^{\nu_0}(-a_j)\prod_{l=1}^{r} \Gamma(\nu_l+i).\label{gL} \ee

According to Proposition \ref{PB}, we must now invert the matrix (\ref{gL}).
With $G=[g_{i,j}]_{i,j=1}^N$, let $C=(G^{-1})^{t}$, the entries $c_{i,j}$ of $C$ then satisfy
\be \label{bv} e^{a_k}\sum_{i=1}^{N}(i-1)!\, L_{i-1}^{\nu_0}(-a_k)\prod_{l=1}^{r} \Gamma(\nu_l+i) \, c_{i,j}=\delta_{j,k}.\ee
Without loss of generality we assume that $a_1, \ldots, a_N$ are pairwise distinct. In this case the above equations imply
 \be \sum_{i=1}^{N}(i-1)!\,L_{i-1}^{\nu_0}(u) \prod_{l=1}^{r} \Gamma(\nu_l+i)\, c_{i,j}=e^{-a_j}\prod_{l=1,l\neq j}^{N} \frac{-u-a_l}{a_j-a_l}, \label{csumidentity}\ee
as can be verified by noting that both sides are polynomials of degree $N-1$ in $u$ which are equal at $N$ different points since (\ref{bv}) is
satisfied. Using this implicit formula for $\{c_{i,j}\}$ we now want to show that (\ref{7.eb2'}) implies the double contour integral
formula (\ref{kernelCD}).

Using the  integral representation of the  reciprocal   Gamma function
  \begin{equation}\label{intrepgamma1}
    \frac{1}{\Gamma(z)} = \frac{1}{2\pi i} \int_{\gamma} w^{-z}e^{w} dw,
  \end{equation}
    we have from  \eqref{7.eb2'}  that
  \begin{align}&K_N(x,y) = \frac{1}{(2\pi i)^{r}}\sum_{i,j=1}^{N}   \xi_j(y)  \int_{\gamma_1}dw_1\cdots \int_{\gamma_r}dw_r \prod_{l=1}^r w_{l}^{-\nu_l-1}e^{w_l}\nonumber \\
  & \quad \times \big(\frac{x}{w_1\cdots w_r}\big)^{i-1} \prod_{l=1}^{r} \Gamma(\nu_l+i) \, c_{i,j} \nonumber
  \\
    &= \frac{1}{(2\pi i)^{r}}\sum_{j=1}^{N}   \xi_j(y)  \int_{\gamma_1}dw_1\cdots \int_{\gamma_r}dw_r \prod_{l=1}^r w_{l}^{-\nu_l-1}e^{w_l} \sum_{i=1}^{N} (i-1)! \prod_{l=1}^{r} \Gamma(\nu_l+i)\, c_{i,j}
   \nonumber\\
  &   \quad\times
  \frac{1}{\Gamma(\nu_0+1)}\,e^{\frac{x}{w_1\cdots w_r}} \int_{0}^{\infty} du \, u^{\nu_0} L_{i-1}^{\nu_0}(u)\, e^{-u}{}_0F_1\big(\nu_0+1;-\frac{ux}{w_1\cdots w_r}\big)   \nonumber\\
  &= \frac{1}{(2\pi i)^{r}}\sum_{j=1}^{N}   \xi_j(y)  \int_{\gamma_1}dw_1\cdots \int_{\gamma_r}dw_r \prod_{l=1}^r w_{l}^{-\nu_l-1}e^{w_l} \ e^{\frac{x}{w_1\cdots w_r}}
   \nonumber\\
  &   \quad\times
  \frac{1}{\Gamma(\nu_0+1)} \int_{0}^{\infty} du \, u^{\nu_0} e^{-u}{}_0F_1\big(\nu_0+1;-\frac{ux}{w_1\cdots w_r}\big)\,e^{-a_j}\prod_{l\neq j} \frac{-u-a_l}{a_j-a_l}. \label{2.23}
  \end{align}
  Here the formulae \eqref{laguerreformula2} and \eqref{csumidentity} have been made use of respectively in the second and third equalities.

  Finally, with \eqref{T} and \eqref{function3}    in mind, these facts substituted into (\ref{2.23}) imply that
\be
  K_N(x,y) = \int_0^\infty du \, u^{\nu_0} e^{-u} \Psi(u;x) \sum_{j=1}^N \Phi(-a_j;y)\, e^{-a_j} \prod_{l\ne j} {- u - a_j \over a_j - a_l}.
\ee
  We recognise the sum over $j$ as the sum of the residues at $\{a_l\}$ of
\be \Phi(-v;y)\frac{1}{-u-v}\prod_{l=1}^{N}\frac{-u-a_l}{v-a_l}\ee
  considered as a function of $v$.
  Applying the residue theorem and   changing $v$ to $-v$, we thus arrive at the desired result.
\end{proof}

\begin{remark} The case in which each $a_l = 0$ has been analysed previously \cite{AIK13, AKW13,KZ}, but using different working. Thus instead of computing the
inverse matrix (\ref{7.eb1'}), functions
$$
P_{j-1}(x) \in {\rm Span} \, \{ \eta_1(x),\dots, \eta_j(x) \}, \qquad
Q_{j-1}(x) \in {\rm Span} \, \{ \xi_1(x),\dots, \xi_j(x) \},
$$
with the biorthogonality property $\int_0^\infty P_k(x) Q_l(x) \, dx = \delta_{k,l}$ were constructed. In terms of these functions (\ref{7.eb2'}) simplifies from
a double sum to the single sum
\begin{equation}\label{SM}
K_N(x,y) = \sum_{j=1}^N P_{j-1}(x) Q_{j-1}(y).
\end{equation}
Instead of (\ref{kernelCD}) with each $a_l=0$ (which strictly speaking is ill-defined
due to the restriction on the contour $\mathcal C$,  but can be well understood in a limiting sense, cf.~Remark \ref{nullremark}) , this leads to the double integral formula
\begin{multline}\label{bq}
 K_N(x,y)  =  {1 \over (2 \pi i)^2} \int_{-1/2 - i \infty}^{-1/2 + i \infty} du  \oint_{\Sigma} dt \,
   \prod_{j=-1}^r { \Gamma(\nu_j + u+1 ) \over   \Gamma(\nu_j + t+1 ) }
   \\
   \times {\Gamma(t - N + 1) \over \Gamma(u - N + 1)} {x^t y^{-(u+1)} \over u - t},
   \end{multline}
where $\Sigma$ is a simple closed contour encircling anti-clockwise $t=0,1,\dots,N-1$ but not $u$; see \cite[Prop.~5.1]{KZ} for the detailed derivation, where similar integral representations for both multiple orthogonal functions $P_{j-1}$ and $Q_{j-1}$ are first derived and then the double integral follows from a particular combination.
\end{remark}

Next, we further  investigate Proposition \ref{pdf} and establish   a corollary      under the assumption that  all but a fixed number of source parameters are equal to $a$. Precisely, for $m\geq 0$ let $a_{m+1}=\cdots=a_N=a$. More definitions are also needed. For $k=1, 2, 3, \ldots$ and $n=0,1,2,\ldots$, set
\be  \mathcal{L}^{(k)}_{n}(x;a,a_1,\ldots,a_{k-1})= \int_{0}^{\infty} du   \, u^{\nu_0}e^{-u} \Psi(u;x) (u+a)^{n}\prod_{l=1}^{k-1} (u+a_l),\label{typeIIfunction}\ee
and \be  \mathcal{\widetilde{L}}^{(k)}_{n}(x;a,a_1,\ldots,a_k)=\frac{1}{2\pi i} \int_{\mathcal{C}} dv \,  e^{v} \Phi(v;x)(v+a)^{-n} \prod_{l=1}^{k}\frac{1}{v+a_l}, \label{typeIfunction}\ee
where $\mathcal{C}$ is a counterclockwise contour encircling $-a, -a_1,\ldots, -a_k$  but not any point on the positive real axis.
\begin{corollary}\label{pdfquasi} Let $K_N$ be the  kernel  \eqref{kernelCD}, and for $m \ge 0$
  let
\be  a_{m+1}= \cdots= a_N=a. \label{arelation}\ee  Then we have
\begin{align}
K_N(x,y)&=\frac{1}{2\pi i}\int_{0}^{\infty} du \int_{\mathcal{C}} dv \, u^{\nu_0}e^{-u+v} \Psi(u;x)\Phi(v;y)\frac{1}{u-v}\Big(\frac{u+a}{v+a}\Big)^{N-m}\nonumber\\
&\quad +\sum_{k=1}^m  \mathcal{L}^{(k)}_{N-m}(x;a,a_1,\ldots,a_{k-1})\, \mathcal{\widetilde{L}}^{(k)}_{N-m}(x;a,a_1,\ldots,a_k), \label{kernelCDquasi}\end{align}
where $\mathcal{C}$ is a counterclockwise contour encircling $-a$ but not $u$.
 \end{corollary}

  \begin{proof} We will use the identity
  \be \frac{1}{u-v}\prod_{l=1}^{m}\frac{u+a_l}{v+a_l}=\frac{1}{u-v}+\sum_{k=1}^m \frac{\prod_{l=1}^{k-1}(u+a_l)}{\prod_{l=1}^{k}(v+a_l)} \label{identicalrelation}\ee
 which has been proved by induction  in \cite{DF06}; see the equation (5.12) therein.  A direct   proof can be given as follows. Rewriting   $u-v=u+a_k-(v+a_k)$, we have
 \begin{align}(u-v)\sum_{k=1}^m \frac{\prod_{l=1}^{k-1}(u+a_l)}{\prod_{l=1}^{k}(v+a_l)}&=\sum_{k=1}^m \prod_{l=1}^{k}\frac{ u+a_l }{ v+a_l} -\sum_{k=1}^m \prod_{l=1}^{k-1}\frac{ u+a_l }{ v+a_l}= \prod_{l=1}^{m}\frac{ u+a_l }{ v+a_l} -1,
 \end{align}
 and \eqref{identicalrelation} follows.

 Recalling \eqref{arelation}, substituting \eqref{identicalrelation} in \eqref{kernelCD},   comparing the sought equation with \eqref{typeIIfunction} and \eqref{typeIfunction} this completes the proof. \end{proof}

\subsection{Average of characteristic polynomials} \label{sectionpolynomial}
Recall that a  biorthogonal ensemble  \cite{Bo98} refers to the joint density function
 \be \mathcal{Q}_{N}(x_1,\ldots,x_N)=\frac{1}{Z_N} \det[\eta_i(x_j)]_{i,j=1}^{N}\det[\xi_i(x_j)]_{i,j=1}^{N}\label{biopdf}, \ee
where all variables $x_1,\ldots, x_N$ are assumed to lie in the same interval $I\subseteq \mathbb{R}$ for simplicity.

For the special case $\eta_i=x^{i-1}$, the average ratio of characteristic polynomials under the density \eqref{biopdf} can be expressed in terms of the correlation kernel; thus as a minor variant
of \cite[Prop.~1]{DF08} we have the following.

\begin{prop} \label{averageratio}
With the same assumption and notation as in Proposition \ref{PB}, let    $\eta_j(x)=x^{j-1}$ for   $j=1, 2, \ldots$. Then,  for $z\in \mathbb{C}\backslash I$
\be \mathbb{E}\big[\prod_{l=1}^{N}\frac{x-x_l}{z-x_l}\big]=\int_I du \, \frac{x-u}{z-u}K_N(x,u).\ee
Equivalently, if for $x\in \mathbb{R}$   we define the residue \be {{\rm Res}_{z=x}} f(z)=\displaystyle\lim_{\varepsilon\rightarrow 0^+} \frac{1}{\pi}{ \rm Im} \, f(x-i\varepsilon), \ee then
\be K_N(x,y)=\frac{1}{x-y}{{\rm Res}_{z=x}\mathbb{E}\big[\prod_{l=1}^{N}\frac{x-x_l}{z-x_l}\big]}.\ee
\end{prop}

In the case of the average of a single characteristic polynomial or its reciprocal, alternative expressions are also available; cf.~Proposition 2 of \cite{DF08}.

\begin{prop}  \label{averageone}
With the same assumption and notation as in Proposition \ref{PB}, let    $\eta_j(x)=x^{j-1}$ for   $j=1, 2, \ldots$. Then,
\begin{align} \mathbb{E}\big[\prod_{l=1}^{N}\frac{1}{z-x_l}\big]&=\frac{N!}{Z_N}
        \begin{vmatrix}
                       g_{1,1} & g_{1,2} & \ldots & g_{1,N} \\
                       \vdots & \vdots & \ddots & \vdots \\
                       g_{N-1,1} & g_{N-1,2} & \ldots & g_{N-1,N} \\
                       \int_I du \, \frac{\xi_1(u)}{z-u} &  \int_I du  \, \frac{\xi_2(u)}{z-u} & \ldots &  \int_I du \, \frac{\xi_N(u)}{z-u} \\
                     \end{vmatrix} \label{multiple1}\\
                     &=\int_I du \, \frac{1}{z-u}\sum_{j=1}^N c_{N,j}\eta_j(u),  \label{multiple2}
                     \end{align}
                     for $z\in \mathbb{C}\backslash I$ and
                     \begin{align} \mathbb{E}\big[\prod_{l=1}^{N}(x-x_l)\big]&=\frac{N!}{Z_N}
        \begin{vmatrix}
                       g_{1,1} &  \ldots & g_{1,N} & \eta_{1}(x) \\
                        g_{2,1} &  \ldots & g_{2,N} & \eta_{2}(x)\\
                       \vdots & \vdots & \vdots & \vdots \\
                       g_{N+1,1}  & \ldots &  g_{N+1,N} & \eta_{N+1}(x)  \label{multiple3}\\
                                            \end{vmatrix} \\
                     &= \frac{1}{\tilde{c}_{N+1,N+1}}\sum_{j=1}^{N+1} \tilde{c}_{j,N+1} x^{j-1}, \label{multiple4}
                     \end{align}
                     where the normalization $Z_N=N! \det[g_{i,j} ]_{i,j=1}^{N}$,  and $\tilde{c}_{j,N+1}$  is the $(N+1,j)$ entry of the inverse of $[g_{i,j}]_{i,j=1}^{N+1}$.
\end{prop}
\begin{proof} The formulas \eqref{multiple1} and \eqref{multiple3} have been proved in Proposition 2 of \cite{DF08}. After noting the facts $[g_{i,j}]_N [c_{i,j}]^{t}_N=I_N$ and $N!/Z_N=\det[c_{i,j}]^{t}_N$, building on \eqref{multiple1} simple manipulation gives  \eqref{multiple2}. On the other hand, by $[\tilde{c}_{i,j}]^{t}_{N+1}[g_{i,j}]_{N+1} =I_{N+1}$ and
\be \frac{N!}{Z_N}=\frac{Z_{N+1}}{(N+1)Z_N} \frac{(N+1)!}{Z_{N+1}}=\frac{Z_{N+1}}{(N+1)Z_N} \det[\tilde{c}_{i,j}]^{t}_{N+1},\ee
we have from \eqref{multiple3} that
\be \mathbb{E}\big[\prod_{l=1}^{N}(x-x_l)\big]=\frac{Z_{N+1}}{(N+1)Z_N}\sum_{j=1}^{N+1} \tilde{c}_{j,N+1} x^{j-1}, \ee which further implies the sought equation  \eqref{multiple4} since it is a monic  polynomial.  \end{proof}

Application of the previous two propositions  gives us explicit evaluation of averages of characteristic polynomials for the   product of random matrices (\ref{Y}).
\begin{prop} \label{averageproduct} For the eigenvalue PDF (\ref{B}), the following hold true.

 (i) Let $K_N$ be the  kernel  \eqref{kernelCD}, then for $z\in \mathbb{C}\backslash\mathbb{R}$,
\be \mathbb{E}\big[\prod_{l=1}^{N}\frac{x-x_l}{z-x_l}\big]=\int_0^{\infty} du \frac{x-u}{z-u}K_N(x,u).\ee

(ii) Let $\Phi$ be given by \eqref{function4}, then for $z\in \mathbb{C}\backslash\mathbb{R}$,
\be \label{NE1} \mathbb{E}\big[\prod_{l=1}^{N}\frac{1}{z-x_l}\big]=\frac{1}{2\pi i}\frac{(-1)^{N-1}}{\prod_{l=1}^r\Gamma(\nu_l+N)}\int_0^{\infty} dt \int_{\mathcal{C}} dv \, e^v \Phi(v;t)\frac{1}{z-t}\prod_{l=1}^{N}\frac{1}{v+a_l}\ee
where $\mathcal{C}$ is a counterclockwise contour encircling $-a_1,\ldots, -a_{N}$ but not any point on the positive real axis.

(iii) Let $\Psi$ be given by \eqref{function3}, then
\be \label{NE2} \mathbb{E}\big[\prod_{l=1}^{N}(x-x_l)\big]=(-1)^{N}\prod_{l=1}^r\Gamma(\nu_l+N+1) \int_0^{\infty} du \, u^{\nu_0} e^{-u}\, \Psi(u;x) \prod_{l=1}^{N}(u+a_l).\ee
\end{prop}

\begin{proof}   It is immediate that Proposition  \ref{averageratio} implies (i). For (ii),  noting that the leading term of the Laguerre polynomial is
\be n!L_{n}^{\nu_0}(x)=(-x)^n+\cdots,\ee
  dividing by $(-u)^{N-1}$ and   taking the limit $u \to \infty$ in (\ref{csumidentity})  we see  that  \be c_{N,j}=\prod_{l=1}^r\frac{1} {\Gamma(\nu_l+N)}\, e^{-a_j}\prod_{l=1,l\neq j}^{N} \frac{1}{a_j-a_l}.\ee
Substituting $c_{N,j}$ in \eqref{multiple2} and noting $\eta_j(u)=\Phi(-a_j;u)$, we obtain (\ref{NE1}).

  For (iii), we first  introduce an auxiliary variable $a_{N+1}$  and set $\eta_{N+1}(u)=\Phi(-a_{N+1};u)$. The fact that $([g_{i,j}]_{N+1})^t [\tilde{c}_{i,j}]_{N+1}=I_{N+1}$ implies
  \be \tilde{c}_{N+1,N+1}=\frac{\det[g_{i,j}]_{N}}{\det[g_{i,j}]_{N+1}}=\prod_{l=1}^r\frac{1} {\Gamma(\nu_l+N+1)}\, e^{-a_{N+1}}\prod_{l=1}^{N} \frac{1}{a_{N+1}-a_l}. \label{cnn1}\ee
Changing $N$ to $N+1$ and using \eqref{csumidentity}, as   derived  in   \eqref{2.23} we obtain
\begin{align} \sum_{j=1}^{N+1} \tilde{c}_{j,N+1} x^{j-1}=\int_0^{\infty} du  u^{\nu_0} e^{-u}\, \Psi(u;x) e^{-a_{N+1}}\prod_{l=1}^{N}\frac{-u-a_l}{a_{N+1}-a_l}.\label{cn1sum}\end{align}
Combination of \eqref{cnn1}, \eqref{cn1sum} and \eqref{multiple4} completes the proof of (iii).
  \end{proof}

  Again, for  the eigenvalue PDF (\ref{B}), let
\be Q_{N-1}(x)=\textrm{Res}_{z=x}\mathbb{E}\big[\prod_{l=1}^{N}\frac{1}{z-x_l}\big], \qquad x\in (0,\infty),\ee
then use of Proposition \ref{averageproduct} (ii)  shows
 \be   Q_{N-1}(x)=  \frac{1}{2\pi i}\frac{(-1)^{N-1}}{\prod_{l=1}^r\Gamma(\nu_l+N)} \int_{\mathcal{C}} dv \, e^v \Phi(v;x) \prod_{l=1}^{N}\frac{1}{v+a_l};\ee
 when  $a_1, \ldots, a_N$ are pairwise distinct it is a special case of  Proposition 2   \cite{DF08}. Also, let
  \be \label{PN}
P_N(x)=\mathbb{E}\big[\prod_{l=1}^{N}(x-x_l)\big],
\ee
  combining Corollary \ref{pdfquasi} where $m$ is taken to be zero and Proposition \ref{averageproduct},   the correlation kernel $K_N$ given by  \eqref{kernelCD} can be expressed as the single sum (\ref{SM}) in terms of $P_j(x)$ and $Q_j(x)$. Here, without loss of generality, it is assumed that $P_j(x)$  corresponds to the   multi-parameters   $a_1, \ldots, a_j$ while $Q_j(x)$ corresponds to $a_1, \ldots, a_{j+1}$.

  \begin{remark}\label{R2.2}

In the special case $r=0$, use of (\ref{E0a}) shows that   (\ref{NE2}) reduces to
$$
P_N(x) = {(-1)^N e^{x} \over \Gamma(\nu_0 + 1) }
\int_0^\infty u^{\nu_0} e^{-u} \, {}_0 F_1 ( \nu_0 + 1; - xu) \prod_{l=1}^N ( u + a_l) \, du.
$$
This same expression has been derived using combinatorial means in \cite{Fo13}, and as the solution
of a partial differential equation in \cite{BNW14}. Furthermore, in this  case   $Q_{N-1}(x)$ and $P_N(x)$ are so-called multiple functions of type I and II respectively, and (\ref{SM}) reduces to Corollary 7 in \cite{DF08}; see \cite{DF08} or \cite{Ku10} for more details, especially when the parameters $a_j$'s coalesce into $D$ different values.  For the case of $a_1=\cdots=a_N=0$ and general $r$,   $Q_{N-1}(x)$ and $P_N(x)$ are also multiple functions of type I and II associated with $r+1$ weights; see \cite{KZ}. However, in the general case it remains as a challenge to  identify a multiple orthogonal functions structure.
\end{remark}

\section{Hard edge limits}\label{sectionhardlimit}
In this section we choose the source $A$ such that all but possibly a fixed number $m$ of the eigenvalues of $A^*A$ are equal to $bN$.
Three regimes are distinguished: subcritical regime $0<b<1$, critical regime $b=1$ and supercritical regime $b>1$; as to the former two regimes, see \cite{KMW09} and \cite{KFW11} for a relevant discussion on non-intersecting Bessel paths which corresponds to the  case $r=0$.
 In the present  paper we focus on the scaled hard edge limits   in the three regimes and leave the bulk and soft-edge  limits to a future work; for the case $a_1=\cdots=a_N=0$ the latter two limits have been established in \cite{LWZ14}. The critical kernel results from a double scaling limit, and its functional form is our main result as stated in Section \ref{sectionintroduction}.
 As $b$ increases from zero,  we will  describe a phase transition from the Meijer G-kernel  (\ref{r})  to the critical kernel (cf. Theorem \ref{criticalkernel}), then to the shifted mean LUE   kernel \eqref{kernelCD} (cf. Theorem \ref{supercriticalkernel} for the case $r=0$).

\subsection{Limiting kernels} \label{limitkernel}
We first suppose that  $0<b<1$. The hard edge scaling in this parameter range is in fact independent of $b$, and the hard edge correlation kernel (\ref{r}) already
known for the case $b=0$  is reclaimed.

\begin{theorem}[Subcritical regime]\label{noncriticalkernel}  With   the   kernel  \eqref{kernelCD},   let
\be a_1= \cdots= a_N=b N.
 \label{confluentsource}\ee
  Then for $0<b<1$,  we have
\be
 \lim_{N\rightarrow \infty}\frac{1}{(1-b)N}K_N\Big(\frac{\xi}{(1-b)N},\frac{\eta}{(1-b)N}\Big)= K^{{\rm h}, r}(\xi, \eta), \label{equationsub}
 \ee
 where $K^{{\rm h}, r}$ is given by (\ref{r}), valid uniformly for $\xi, \eta$ in any compact set  of $(0,\infty)$.
 \end{theorem}

 \begin{proof}  Introducing rescaled variables   $x= \xi/((1-b)N), \,  y= \eta/((1-b)N)$ in \eqref{kernelCD}
 and substituting $u, v$  by  $uN, vN$ respectively, we obtain
 \begin{multline}   \frac{1}{(1-b)N}K_N\Big(\frac{\xi}{(1-b)N},\frac{\eta}{(1-b)N}\Big)=\int_{0}^{\infty}du \int_{\mathcal{C}} dv
  \\  \times \frac{1}{2(1-b)\pi i}
     \frac{e^{-N(f(u)-f(v))}}{u-v}
   (Nu)^{\nu_0} \Psi\big(Nu; \frac{\xi}{(1-b)N}\big) \Phi\big(Nv; \frac{\eta}{(1-b)N}\big),
    \label{rescalingkernel}\end{multline}
     where
     \be f(z)=z-\log(b+z). \ee
Although both the functions $\Psi$ and $\Phi$ in the integrand  of \eqref{rescalingkernel} depend on $N$, we will see that for the large $N$ they do not enter the saddle point equation (e.g., cf. \eqref{phiasymtotics} and \eqref{psiasymtotics} below).   So we may perform saddle-point approximations and this is what we will do next.

     Consider now the exponent on the RHS of \eqref{rescalingkernel}. Since \be f'(z)=1-\frac{1}{b+z}, \ee
     there is a saddle point $z_0=1-b$. We   hereby deform the contour $\mathcal{C}$ into the union of two closed contours $\mathcal{C}_1 \bigcup \mathcal{C}^{-}_2$  such that   $\mathcal{C}_{1} =\{z\in \mathbb{C}:\abs{z+b}=1\}$  and  $\mathcal{C}^{-}_2$ is a clockwise
 contour encircling the segment $[0,1-b]$ but not $-b$ ($\mathcal{C}^{-}_2$  and $\mathcal{C}_2$ refer to the same curve except that the former indicates the clockwise direction). For instance, we can choose $\mathcal{C}_2$ as   the union of two segments from $-0.5b$ to $ -b+e^{\pm i\epsilon}$ respectively and an arc $\{z: z=-b+e^{i\theta},-\epsilon \leq \theta \leq \epsilon \}$ for some small positive $\epsilon$. With such a choice, we   divide the integration over $\mathcal{C}$ into two parts, and furthermore rewrite  the double   integral on the RHS of  \eqref{rescalingkernel} as a sum of two integrals, that is, \begin{align}   \frac{1}{(1-b)N}K_N\Big(\frac{\xi}{(1-b)N},\frac{\eta}{(1-b)N}\Big)&=
 \textrm{p.v.}\int_{0}^{\infty}du \int_{\mathcal{C}_1} dv  (\cdot)+ \textrm{p.v.}\int_{0}^{\infty}du \int_{\mathcal{C}^{-}_2} dv  (\cdot)\nonumber\\
 &:=I_{1}+I_{2}.
   \label{twointegrals}\end{align}
  Here the notation $\textrm{p.v.}$ denotes the Cauchy principal value  integral.  It is worth stressing that,   from \eqref{rescalingkernel},   we can put some restrictions on  the range of $u, v$ in the above integrals such that $u\neq 1-b$ and $v\neq \pm i$. This is done for   the convenience of  subsequent asymptotic analysis only.

   As   $N\rightarrow \infty$,  we claim that  the leading  contribution of the double   integral on the RHS of \eqref{rescalingkernel} comes from the range of $u\in (0,1-b)$ and $v\in \mathcal{C}_2$. Actually, for $I_2$, when $u>1-b$  the $v$-integral vanishes by Cauchy's theorem since the integrand does not have any singularities inside $\mathcal{C}_2$, while for $0<u<1$  application of the residue theorem gives
   \begin{align}   &I_2=\frac{1}{1-b}\int_{0}^{1-b}du\,
   (Nu)^{\nu_0} \Psi\big(Nu; \frac{\xi}{(1-b)N}\big) \Phi\big(Nu; \frac{\eta}{(1-b)N}\big). \label{I2form}
  \end{align}

   Using the asymptotic expansion of the function
   ${}_1F_1$ for the large argument (cf.~Theorem 4.2.2 and Corollary 4.2.3, \cite{AAR}), for large  $N$  we have
   \be  {}_1F_1(\nu_0+s;\nu_0+1;-Nv)= \frac{\Gamma(\nu_0+1)}{\Gamma(1-s)} (Nv)^{-\nu_0-s}\big(1+\mathcal{O}(\frac{1}{N})\big), \  \textrm{Re}\, v>0, \label{1asyptotics1f1}\ee
and
\be
  {}_1F_1(\nu_0+s;\nu_0+1;-Nv)=    \frac{\Gamma(\nu_0+1)}{\Gamma(\nu_0+s)}  (-Nv)^{s-1}e^{-Nv}\big(1+\mathcal{O}(\frac{1}{N})\big), \  \textrm{Re}\, v<0.  \label{2asyptotics1f1}\ee
  Keeping in mind  \eqref{function4} and \eqref{function1}, by definition of the Meijer G-function  ~\eqref{Gfunction}  we have from \eqref{1asyptotics1f1} that \be  (Nu)^{\nu_0}  \Phi\big(Nu; \frac{\eta}{(1-b)N}\big)\sim
    G^{r+1,0}_{0,r+2} \Big ({ \underline{\hspace{0.5cm}}
 \atop \nu_0,  \dots,\nu_r,0} \Big | \frac{u \eta}{1-b} \Big ).\label{phiasymtotics}\ee
 Here and below we use the notation $f_N \sim g_N$ to mean that $\lim_{N \to \infty} f_N/g_N = 1$.
 On the other hand, consideration of the definition \eqref{function3} shows
  \be   \Psi\big(Nu; \frac{\xi}{(1-b)N}\big)\sim G^{1,0}_{0,r+2} \Big ({ \underline{\hspace{0.5cm}}
 \atop 0,-\nu_0, \dots,-\nu_r} \Big |  \frac{u \xi}{1-b} \Big ), \label{psiasymtotics}\ee
 where use has been made of the identity (cf. eqn~ \eqref{FGrelation})
 \be G^{1,0}_{0,r+2} \Big ({ \underline{\hspace{0.5cm}}
 \atop 0,-\nu_0, \dots,-\nu_r} \Big |  z \Big )=  \prod_{l=0}^{r}\frac{1}{\Gamma(\nu_l+1)}\, {}_0F_{r+1}\big(\nu_0+1,  \dots,\nu_r+1; -z\big).\label{fgrelation}\ee

 Combining \eqref{I2form}, \eqref{phiasymtotics} and \eqref{psiasymtotics}, and changing variables we get
 \be I_2 \rightarrow  K^{{\rm h}, r}(\xi, \eta). \label{I2term}\ee

 Next,  we   deal with the integral $I_1$ and show that it is negligible. In this case because of different asymptotic forms of ${}_1F_1$ given in \eqref{1asyptotics1f1} and \eqref{2asyptotics1f1}, we divide $I_1$ into two parts as
 \begin{align}   I_1=  \textrm{p.v.}\int_{0}^{\infty}du \int_{\mathcal{C}_{1,+}} dv  (\cdot)+  \int_{0}^{\infty}du \int_{\mathcal{C}_{1,-}} dv  (\cdot):=I_{11}+I_{12},
   \label{twosubintegrals}\end{align}
 where $\mathcal{C}_{1,+}=\mathcal{C}_1 \bigcap\{v:\textrm{Re}\, v>0\}$ and $\mathcal{C}_{1,-}=\mathcal{C}_1 \bigcap\{v:\textrm{Re}\, v<0\}$.
 Notice that for $0<b<1$ one can easily check  that
    $ \textrm{Re}\{f(u)\}$   attains its global minimum at $u=1-b$ over $(0,\infty)$, while $ \textrm{Re}\{f(v)\}$    attains its global maximum at $v=1-b$ over $\mathcal{C}_1$.   Therefore,  for $I_{11}$ combining  \eqref{function4}, \eqref{function1}, \eqref{1asyptotics1f1} and  \eqref{psiasymtotics} we have
    \begin{multline}    I_{11}\sim  \textrm{p.v.} \int_{0}^{\infty}du \int_{\mathcal{C}_{1,+}} dv
   \frac{1}{2(1-b)\pi i}
     \frac{e^{-N(f(u)-f(v))}}{u-v}    \big( \frac{u}{v}\big)^{\nu_0} \\ G^{1,0}_{0,r+2} \Big ({ \underline{\hspace{0.5cm}}
 \atop 0,-\nu_0,\dots,-\nu_r} \Big |  \frac{u \xi}{1-b} \Big )\, G^{r+1,0}_{0,r+2} \Big ({ \underline{\hspace{0.5cm}}
 \atop \nu_0,  \dots,\nu_r,0} \Big | \frac{v \eta}{1-b} \Big ).
 \end{multline}
   For this,   the standard steepest  descent argument shows that the main contribution   comes from the neighbourhood of the saddle point $z_0=1-b$, namely,
    \be I_{11}=\mathcal{O}(1/\sqrt{N}). \label{I11term}\ee
     Similarly, for  $I_{12}$ combination of \eqref{function4}, \eqref{function1}, \eqref{2asyptotics1f1} and \eqref{psiasymtotics} then gives us
    \begin{align}    &I_{12}\sim     \int_{0}^{\infty}du \int_{\mathcal{C}_{1,-}} dv
   \frac{1}{2(1-b)\pi i} \frac{e^{-N(f(u)-f(1-b))}}{u-v}
       G^{1,0}_{0,r+2} \Big ({ \underline{\hspace{0.5cm}}
 \atop 0,-\nu_0,\dots,-\nu_r} \Big |  \frac{u \xi}{1-b} \Big )   \nonumber \\ &    \times   u^{\nu_0} e^{-N(\log(b+v) +1-b)}
   \frac{1}{2\pi i} \int_{c-i\infty}^{c+i\infty} ds\, \big(\frac{\eta}{1-b}\big)^{-s} (-v)^{s-1}N^{\nu_0+2s-1}   \prod_{l=1}^{r} \Gamma(\nu_l+s),
 \end{align}
  for which the  integrals of $u$ and $v$  respectively afford us  bounds $\mathcal{O}(1/\sqrt{N})$ and  $\mathcal{O}(N^{\nu_0+2c-1}e^{-(1-b)N})$. Together,  we obtain  the exponential decay estimation
    \be I_{12}=\mathcal{O}(N^{\nu_0+2c-3/2}e^{-(1-b)N}). \label{I12term}\ee

    Combining  \eqref{I2term}, \eqref{I11term} and \eqref{I12term}, we arrive at the equation \eqref{equationsub}. Furthermore, it is clear that the previously derived  estimates are valid  uniformly for $\xi, \eta$ in any given compact set  of $(0,\infty)$.
   \end{proof}

\begin{remark} When all the  parameters $a_l$'s are null, if we understand   the  double   integral representation \eqref{kernelCD}    as described  in  Remark \ref{nullremark}, then  the same  argument   as  in the proof of Theorem \ref{noncriticalkernel} is also applicable. This gives another derivation of
(\ref{r}) different from  that in \cite{KZ}.  \end{remark}

We turn to proofs of Theorems \ref{criticalkernel} and \ref{deformedcriticalkernel}.
 \begin{proof}[Proof of Theorem  \ref{criticalkernel}]    Rescaling  variables     in \eqref{kernelCD},    we have
 \begin{multline}   \frac{1}{ \sqrt{N}}K_N\Big(\frac{\xi}{\sqrt{N}},\frac{\eta}{\sqrt{N}}\Big)=\\\frac{\sqrt{N}}{2\pi i}\int_{0}^{\infty} du \int_{\mathcal{C}} dv \, \frac{e^{-N(f(u)-f(v))}}{u-v} (Nu)^{\nu_0} \Psi(Nu;\frac{\xi}{\sqrt{N}}) \Phi(Nv;\frac{\eta}{\sqrt{N}}),\label{ctriticalCD}
  \end{multline}
     where $ f(z)=z-\log(1+z/b)$ with $b=(1- \tau/\sqrt{N} )^{-1}$.  If $b$ is equal to the critical value 1, then the saddle point  of $f(z)$ is $z_0=0$. This time, for a fixed small number $\delta>0$
     we  choose the contour as \be   \mathcal{C} =\{z=-1+(1+2\delta)e^{i\theta}:\theta_0 \leq \abs{\theta}\leq \pi\}\cup \mathfrak{L}_{\mathrm{A_{-}} \mathrm{O}\mathrm{A_{+}}},\ee
     where   \be
      \theta_0=\arccos\frac{1+\delta}{1+2\delta}
, \qquad A_{\pm}=-1+(1+2\delta)e^{\pm i\theta_{0}},\ee
     and $\mathfrak{L}_{\mathrm{A_{-}} \mathrm{O}\mathrm{A_{+}}}$ denotes the union of two line segments from  the point $\mathrm{A_{-}}$ to the origin to the point  $\mathrm{A_{+}}$.
It is clear that $A_{\pm}= (\delta, \pm\sqrt{(2+3\delta)\delta})$,  and the   intersections of  the $y$-axis and the contour $\mathcal{C}$  are $B_{\pm}=(0, \pm 2\sqrt{(1+\delta)\delta})$.
 Moreover,  the four points come close to the origin as $\delta\rightarrow 0$, which permits us to use the Taylor series expansion    of $f(v)$ for any $v\in \mathcal{C_{+}}$ defined below \eqref{twopartssum}.

  First, we divide the integral on the RHS of \eqref{ctriticalCD} into two parts
  \be    \frac{1}{ \sqrt{N}}K_N\Big(\frac{\xi}{\sqrt{N}},\frac{\eta}{\sqrt{N}}\Big)=  \int_{0}^{\infty} du \int_{\mathcal{C_{-}}} dv \,(\cdot) + \int_{0}^{\infty} du \int_{\mathcal{C_{+}}} dv \,(\cdot) :=I_{-}+I_{+} ,\label{twopartssum}
  \ee  where   $\mathcal{C}_{-}= \{v\in \mathcal{C}:\textrm{Re}\, v<0\}$ and $\mathcal{C}_{+}=\{v\in \mathcal{C}:\textrm{Re}\, v>0\}$.
     We claim that  the dominant contribution to  \eqref{ctriticalCD}    comes   from the neighbourhoods of   $u_0=0$ and $v_0=0$, so we need to expand the function $f(z)$ at $z_0=0$. With the double scaling  in mind, we obtain the Taylor series  \begin{align} f(z)
    =\frac{\tau z}{\sqrt{N}}+\frac{1}{2} (1-\frac{\tau}{\sqrt{N}})^{2}z^{2}-\frac{1}{3} (1-\frac{\tau}{\sqrt{N}})^{3}z^{3}+\cdots. \label{Taylorexpansion}\end{align}
  Therefore for $I_+$, combining \eqref{function3},  \eqref{function4}, \eqref{function1} and \eqref{1asyptotics1f1}, together with the relation \eqref{fgrelation} and the definition of Meijer G-function  ~\eqref{Gfunction} we see that
  \begin{multline}    I_{+}\sim    \frac{\sqrt{N}}{2\pi i}\int_{0}^{\infty}du \int_{\mathcal{C}_{+}} dv
        \frac{e^{-N(f(u)-f(v))}}{u-v}    \big( \frac{u}{v}\big)^{\nu_0} \\ G^{1,0}_{0,r+2} \Big ({ \underline{\hspace{0.5cm}}
 \atop 0,-\nu_0,\dots,-\nu_r} \Big |   \sqrt{N}u \xi \Big )  G^{r+1,0}_{0,r+2} \Big ({ \underline{\hspace{0.5cm}}
 \atop \nu_0,  \dots,\nu_r,0} \Big |  \sqrt{N}v \eta \Big ). \label{I2critical}
 \end{multline}
 Fix the two endpoints $B_{\pm}$ of $\mathcal{C}_{+}$ and deform it  to the imaginary axis, then after substituting \eqref{Taylorexpansion} into \eqref{I2critical} and rescaling $u,v$ by   $u/\sqrt{N}, v/\sqrt{N} $, we conclude that  $I_+$ converges to  the kernel defined by  \eqref{criticalk}, uniformly for  $\xi, \eta$ in a compact set  of $(0,\infty)$ and for $\tau$ in a compact set  of $\mathbb{R}$.

 Secondly, for the integral $I_-$,  combination of  \eqref{function3},  \eqref{function4}, \eqref{function1} and \eqref{2asyptotics1f1} yields
  \begin{multline}    I_{-}\sim    \frac{1}{2\pi i}\int_{0}^{\infty}du \int_{\mathcal{C}_{-}} dv
        \frac{e^{-Nf(u)-N\log(1+v/b)}}{u-v}    u^{\nu_0}   G^{1,0}_{0,r+2} \Big ({ \underline{\hspace{0.5cm}}
 \atop 0,-\nu_0,\dots,-\nu_r} \Big |   \sqrt{N}u \xi \Big ) \\    \times
   \frac{1}{2\pi i} \int_{c-i\infty}^{c+i\infty} ds\,  \eta^{-s} (-v)^{s-1}N^{\nu_0+(3s-1)/2}   \prod_{l=1}^{r} \Gamma(\nu_l+s).
 \end{multline}
 Since for sufficiently large $N$, \be\textrm{Re}\{\log(1+\frac{v}{b})\} =\frac{1}{2}\log\big((1+2\delta)^2+\mathcal{O}(\frac{\tau}{\sqrt{N}})\big)>\log(1+\delta)\ee
holds true uniformly  for $\tau$ in a compact set of $\mathbb{R}$ and for $v\in \mathcal{C}_{-}$,  use of the steepest descent argument  leads to an exponential decay \be I_-=\mathcal{O}\big(N^{\nu_0-1+1.5c}e^{-N\log(1+\delta)}\big).\ee

Lastly, by combining the foregoing results for $I_-$ and $I_+$, we then complete the proof.
\end{proof}

%

 \begin{proof} [Proof  of Theorem  \ref{deformedcriticalkernel}]    Rescaling  variables     in \eqref{kernelCD},    we have
 \begin{multline}   \frac{1}{ \sqrt{N}}K_N\Big(\frac{\xi}{\sqrt{N}},\frac{\eta}{\sqrt{N}}\Big)=\frac{1}{2\pi i}\int_{0}^{\infty} du \int_{\mathcal{C}} dv\,  N^{\nu_0+1/2}u^{\nu_0}\frac{e^{-N(f(u)-f(v))}}{u-v} \\ \,\times \prod_{j=1}^{m}\frac{(v+b)(u+\sigma_j/\sqrt{N})}{(u+b)(v+\sigma_j/\sqrt{N})} \,  \Psi(Nu;\frac{\xi}{\sqrt{N}})\Phi(Nv;\frac{\eta}{\sqrt{N}}),
  \end{multline}
     where $ f(z)=z-\log(1+z/b)$ with $b= (1-\tau/\sqrt{N})^{-1}$.

  Proceeding as in the proof  of Theorem \ref{criticalkernel}, Taylor  expanding  $f(z)$ at $z=0$, and rescaling $u,v$ by   $u/\sqrt{N}, v/\sqrt{N} $, we  can complete the proof.
 \end{proof}

We next consider the  supercritical case, that is $b>1$. For $r=0$, the limiting eigenvalue density has support $[L_1, L_2]$ with $L_1>0$ (thus the left-most end changes from the hard to the  soft edge as $b$ increases beyond unity as already remarked in the Introduction); see e.g. \cite{KMW09,KFW11}. However, for $r>0$ and fixed $\nu_0, \nu_1,  \ldots, \nu_r \geq 0$, considerations from free probability theory
suggest that the support will include the origin for general $b$. Nonetheless, in the simplest case of $r=0$ a particular tuning and scaling of the supercritical case can be given which, on an appropriate length scale, effectively separates a bunch of  eigenvalues near the origin from the rescaled left-end support.   A similar result is conjectured to be true for the general  $r>0$.

\begin{theorem} [Supercritical regime for $r=0$]\label{supercriticalkernel}
 With    the kernel    \eqref{kernelCD} where $r=0$, for a fixed positive integer $m$  let
\be a_j=\sigma_j b/(b-1), \,  j=1, \ldots, m \  \mbox{and} \  a_k=bN, \, k=m+1, \ldots, N,
 \label{mconfluentsource2}\ee
 where $b>1$ and $\sigma_1, \ldots, \sigma_m>0$. Then   we have
\begin{align}  &\lim_{N\rightarrow \infty}    e^{(1-\frac{1}{b})(\eta-\xi)}    \big(1-\frac{1}{b}\big)   K_N\Big(\big(1-\frac{1}{b}\big) \xi, \big(1-\frac{1}{b}\big)\eta\Big)  =
  \frac{1}{2\pi i}\frac{\eta^{\nu_0}}{(\Gamma(\nu_0+1))^{2}}   \nonumber \\  & \times \,
 \int_{0}^{\infty}du \int_{ \mathcal{C}} dv\, {}_0F_1(\nu_0+1;-u\xi) \, {}_0F_1(\nu_0+1;-v\eta) u^{\nu_0} \frac{e^{-u + v}}{u-v} \prod_{j=1}^{m}\frac{u+\sigma_j}{v+\sigma_j},   \label{r0super}
 \end{align}
where  $\mathcal{C}$ is a counterclockwise contour encircling  $-\sigma_1, \ldots, -\sigma_m$  but not $u$.

 \end{theorem}

 \begin{proof}   Set $\kappa=(b-1)/b$.  For the large $N$,   we have from   \eqref{kernelCD} with $r=0$ that
 \begin{multline}     \big(1-\frac{1}{b}\big) K_N\Big(\big(1-\frac{1}{b}\big) \xi, \big(1-\frac{1}{b}\big)\eta\Big)  =  \frac{\kappa}{2\pi i}\int_{0}^{\infty}du \int_{ \mathcal{C}}  dv \  u^{\nu_0}\frac{e^{-u + v}}{u-v}      \\  \times  \Psi(u;\kappa\xi)\Phi(v;\kappa\eta)  \prod_{j=1}^{m}\frac{u+\sigma_j/\kappa}{v+\sigma_j/\kappa}\,  \Big(\frac{1+u/(bN)}{1+v/(bN)}\Big)^{N-m}:=I_1+I_2,
  \end{multline}
  where we have rewritten   $\mathcal{C}=\mathcal{C}_1\cup\mathcal{C}_2$ and $I_j=\int_{0}^{\infty}du \int_{ \mathcal{C}_j}  dv \,(\cdot)$. The closed contour $\mathcal{C}_1$ encircles  $-\sigma_1/\kappa, \ldots, -\sigma_m/\kappa$  and on its left lies the path  $\mathcal{C}_2$ encircling  $-bN$, beginning at and returning to $-\infty$. Keeping    \eqref{E0a} and \eqref{E1a} in mind,     for the choice of the contour of   $I_2$   we have used   the asymptotic property of ${}_0F_1$ (cf. \cite[Sect. 5.11.2]{Luke})
  \begin{multline} {}_0F_1(\nu_0+1;z)=\frac{\Gamma(\nu_0+1)}{2\sqrt{\pi}}(-z)^{- \frac{1+2\nu_0 }{4} }\\ \times \Big(e^{-2i\sqrt{-z}}(1+\mathcal{O}(\frac{1}{\sqrt{z}}))+e^{2i\sqrt{-z}}(1+\mathcal{O}(\frac{1}{\sqrt{z}}))\Big), \qquad \abs{z} \rightarrow \infty, \label{asymptotic0f1function}\end{multline}

   Again with \eqref{E0a} and \eqref{E1a} in mind,  by taking the limit and changing variables, as $N\rightarrow \infty$ it is less difficult to know that $e^{\kappa(\eta-\xi)} I_1$ goes to the desired integral.  For the part $I_2$, by   the fact \eqref{asymptotic0f1function},
   taking the limit in the integrand  we see that    the $v$-integral over the closed contour $\mathcal{C}_2$   equals zero since the resulting integrand has no pole.   The proof is thus completed.
   \end{proof}


By comparison with (\ref{kernelCD}) in the case $r=0$, $N=m$, $\{a_l\} = \{\sigma_l\}$,   substitution of (\ref{E0a}) and (\ref{E1a}) shows that
the RHS of (\ref{r0super}) is equal to $e^{\eta-\xi}K_m(\xi, \eta) |_{\{a_l\} = \{\sigma_l\}}$, which is equivalent to the kernel for the $m\times m$ Laguerre Unitary Ensemble  with a source (see e.g. \cite{DF08}  or \cite[Chapter 11]{Fo10}).
For general $r \ge 1$,  as to  the supercritical case of $b>1$  the following similar result is expected to be true
\begin{align}  \lim_{N\rightarrow \infty}\big(1-\frac{1}{b}\big) K_N\Big(\big(1-\frac{1}{b}\big) \xi, \big(1-\frac{1}{b}\big)\eta\Big) & =   \frac{1}{2\pi i}\int_{0}^{\infty}du \int_{ \mathcal{C}}  dv \  \big(\frac{u}{\kappa}\big)^{\nu_0}\frac{e^{-   u + v}}{u-v} \nonumber \\  \times \, \,  \Psi(u/\kappa;\kappa\xi)\Phi(v/\kappa;\kappa\eta) \prod_{j=1}^{m}\frac{u+\sigma_j}{v+\sigma_j}\,
 & =: {\widetilde{\mathcal{K}}}^{{\rm h},r}_{m}(\xi,\eta;\kappa;\sigma), \label{supercriticalk}\end{align}
where $\kappa=(b-1)/b$, $\Psi, \Phi$ are given  by \eqref{function3}, \eqref{function4}, and  $\mathcal{C}$ is a counterclockwise contour encircling  $-\sigma_1, \ldots, -\sigma_m$  but not $u$. To prove it, if  we might control the behavior of $e^{-\kappa u}\Psi(u;\kappa\xi)$ and  $e^{\kappa v} \Phi(v;\kappa\eta)$ (for instance, we can try to  derive  an estimate
  $e^{\kappa v} \Phi(v;\kappa\eta)=\mathcal{O}(v^{-2}) $ as $v\rightarrow -\infty$  which  should further be expected to vanish  sub-exponentially), then  $I_2\rightarrow 0$ and $I_1$ goes to the desired integral as in the proof of Theorem \ref{supercriticalkernel}.  But such an estimate is yet to be found. Furthermore,
we expect the correlations implied by (\ref{supercriticalk}) to be the same as those for $K_m(\xi, \eta) |_{\{a_l\} = \{\sigma_l\}}$  after being multiplied by the factor $g(\kappa;\eta)/g(\kappa;\xi)$ for some properly chosen function $g$.   However, the mechanism which makes this true
 in the cases $r \ge 1$ remains to be clarified.

\begin{remark}
It is of interest to contrast the scalings of $\{a_j\}$ in  Theorem  \ref{deformedcriticalkernel} and \ref{supercriticalkernel} applying to the critical and supercritical cases
respectively. Some insight as to the chosen values is possible by restricting attention to the case $r=0$, for which the squared singular values
have the interpretation as non-intersecting Brownian particles confined to a half line, as mentioned   in the Introduction. In this interpretation, the initial
position of particle $j$ is $a_j$, and the particles evolve for time $t=1$. We interpret the values in $a_j$ in
Theorem  \ref{deformedcriticalkernel}  as being such that the particles at the hard edge are all of the same order, with the $k$ outlier
particles appropriately merging with the spectrum edge of the $N-k$ particles which started originally at
$N(1-\tau/\sqrt{N})$. On the other hand,
in Theorem \ref{supercriticalkernel} only the $k$ particles starting at order unity from the hard edge are at order unity from the hard
edge when $t=1$, with the remaining $N-k$ particles never reaching the hard edge by this time.
\end{remark}

\begin{remark}
  If we strengthen the results in Theorems  \ref{criticalkernel}, \ref{deformedcriticalkernel}, \ref{noncriticalkernel} and \ref{supercriticalkernel}
  from uniform convergence into the trace norm convergence of the integral operators with respect to the correlation kernels, then as a direct consequence we  have   the limiting gap probabilities after rescaling, especially including the  smallest eigenvalue distribution; see \cite[Chapters 8 \& 9]{Fo10}.   Since the proof of trace norm convergence is only a technical elaboration that confirms a well-expected result, we do not give the details.
\end{remark}

\subsection{Characteristic polynomials} \label{limitpolynomial} In this subsection we want to evaluate scaling limits for the ratio of  characteristic polynomials according to three different regimes.  \begin{theorem}   With  the eigenvalue PDF (\ref{B}), fix  $m\in \{0,1,2,\ldots,\}$ and let  \be a_{m+1}= \cdots= a_N=N b.\nonumber\ee

 (i) Set  $a_j=N b_j$ with $b_j>0$ for   $j=1,\ldots,m$, if $0<b<1$, then    for $\zeta\in \mathbb{C}\backslash\mathbb{R}$,
\be\lim_{N\rightarrow \infty} \frac{1}{(1-b)N}\mathbb{E}\Big[\prod_{l=1}^{N}\frac{x_l-\xi/((1-b)N)}{x_l-\zeta/((1-b)N)}\Big]=\int_0^{\infty} du \frac{\xi-u}{\zeta-u}K^{{\rm h}, r}(\xi, u).\label{ratioi}\ee

(ii) Set $ a_j=\sqrt{N}\sigma_j$ with $\sigma_j>0$ for $j=1, \ldots, m$, if  $b=1/(1- \tau/\sqrt{N})$ with $\tau\in \mathbb{R}$, then   for $\zeta\in \mathbb{C}\backslash\mathbb{R}$,
\begin{align} \lim_{N\rightarrow \infty} \frac{1}{\sqrt{N}}\mathbb{E}\Big[\prod_{l=1}^{N}\frac{x_l-\xi/\sqrt{N}}{x_l-\zeta/\sqrt{N}}\Big]=\int_0^{\infty} du \frac{\xi-u}{\zeta-u} {\mathcal K}^{{\rm h},r}_{m}(\xi,u;\tau,\sigma).
 \end{align}

(iii) Set $ a_j= \sigma_j b/(b-1)$ with $\sigma_j>0$ for $j=1, \ldots, m$, if  $b>1$ and $m\geq 1$, then  for $r=0$ and  for $\zeta\in \mathbb{C}\backslash\mathbb{R}$,
\begin{align} \lim_{N\rightarrow \infty} (1-\frac{1}{b})\mathbb{E}\Big[\prod_{l=1}^{N}\frac{x_l-(1-\frac{1}{b})\xi}{x_l-(1-\frac{1}{b})\zeta}\Big]=\int_0^{\infty} du \frac{\xi-u}{\zeta-u} {\widetilde{\mathcal{K}}}^{{\rm h},0}_{m}(\xi,u;1-\frac{1}{b};\sigma).
 \end{align}
\end{theorem}

\begin{proof} By Proposition \ref{averageproduct}, Theorems \ref{noncriticalkernel}, \ref{deformedcriticalkernel} and \ref{supercriticalkernel} imply the sought results
although a minor modification in the proof of Theorem \ref{noncriticalkernel} is required in relation to \eqref{ratioi} (in the same circumstance the limiting subcritical kernel still holds true). \end{proof}

Likewise, based on Proposition \ref{averageproduct}, we can prove the following   theorem  concerning the average of one single characteristic polynomial or its inverse. For this
purpose we introduce four sets of  generalised multiple functions (we say generalised since only for $r=0$ do we know the multiple polynomial system;
recall Remark \ref{R2.2}) of types II and I with $m$ parameters  $\sigma_1, \ldots, \sigma_m>0$. For $k=1,2\ldots, m$,   we define two sets of  generalised multiple functions by
\begin{multline}  \Gamma^{(k)}(x;\sigma_1,\ldots,\sigma_{k-1})= \int_{0}^{\infty}du u^{\nu_0} e^{- \tau u-\frac{1}{2}u^2  } \\
\times G_{0,r+2}^{1,0}  \Big({\atop 0, -\nu_0, -\nu_1,\ldots,-\nu_r}\Big|x u \Big)\prod_{j=1}^{k-1}(u+\sigma_j) ,\label{typeIIpearceylikefunction}\end{multline}
and
\begin{multline} \widetilde{\Gamma}^{(k)}(x;\sigma_1,\ldots,\sigma_k)=\frac{1}{2\pi i} \int_{-i\infty}^{i\infty} dv   v^{-\nu_0}e^{\tau v+\frac{1}{2}v^2} \\ \times G_{0,r+2}^{r+1,0} \Big({\atop \nu_0,  \nu_1,\ldots,\nu_r,0}\Big|xv\Big) \prod_{j=1}^{k}\frac{1}{v+\sigma_j}, \label{typeIpearceylikefunction}\end{multline}
while for $0<\kappa\leq 1$ two sets of Laguerre-like  generalised multiple functions are defined by
\be  \mathcal{L}^{(k)}(x;\kappa;\sigma_1,\ldots,\sigma_{k-1})= \int_{0}^{\infty} du   \, (u/\kappa)^{\nu_0}e^{-u} \Psi(u/\kappa;\kappa x) \prod_{l=1}^{k-1} (u+\sigma_l),\label{typeIIlaguerrelikefunction}\ee
and
\be  \widetilde{\mathcal{L}}^{(k)}(x;\kappa;\sigma_1,\ldots,\sigma_k)=\frac{1}{2\pi i} \int_{\gamma} dv \,  e^{v} \Phi(v/\kappa;\kappa x)  \prod_{l=1}^{k}\frac{1}{v+\sigma_l}. \label{typeIlaguerrelikefunction}\ee
Here   $\gamma$  is a closed path which is     encircling $-\sigma_1, \ldots, -\sigma_m$  once in the positive direction.

\begin{theorem}   With  the eigenvalue PDF (\ref{B}), fix  $m\in \{0,1,2,\ldots,\}$ and let  \be a_{m+1}= \cdots= a_N=N b.\nonumber\ee

 (i) Set  $a_j=N b_j$ with $b_j>0$ for   $j=1,\ldots,m$, if $0<b<1$, then    
\begin{multline}\lim_{N\rightarrow \infty} \frac{-\sqrt{N}}{\Upsilon_{N}^{(\mathrm{sub})}}\mathbb{E}\Big[\prod_{l=1}^{N}\frac{1}{x_l-\zeta/((1-b)N)}\Big]\\=\int_0^{\infty} du \frac{1}{\zeta-u}G_{0,r+2}^{r+1,0} \Big({\atop \nu_0,  \nu_1,\ldots,\nu_r,0}\Big|u\Big),\end{multline}
\begin{multline}\lim_{N\rightarrow \infty}  \sqrt{N}\prod_{l=1}^r(\nu_l+N)\, \Upsilon_{N}^{(\mathrm{sub})}  \mathbb{E}\Big[\prod_{l=1}^{N}\big(x_l-\xi/((1-b)N)\big)\Big]\\= G_{0,r+2}^{1,0} \Big({\atop 0,-\nu_0,  -\nu_1,\ldots,\nu_r}\Big|\xi\Big)\end{multline}
where \be \Upsilon_{N}^{(\mathrm{sub})}=(-1)^{N}\sqrt{2\pi}N^{\nu_0+N}e^{-(1-b)N}\prod_{l=1}^r\Gamma(\nu_l+N)\,(1-b)^{\nu_0}\prod_{j=1}^m (1-b+b_j).\ee
(ii) Set $ a_j=\sqrt{N}\sigma_j$ with $\sigma_j>0$ for $j=1, \ldots, m$, if  $b=1/(1- \tau/\sqrt{N})$ 
, then
\begin{equation}\lim_{N\rightarrow \infty} \frac{-\sqrt{N}}{\Upsilon_{N}^{(\mathrm{cri})}}\mathbb{E}\Big[\prod_{l=1}^{N}\frac{1}{x_l-\zeta/\sqrt{N}}\Big]=\int_0^{\infty} du \frac{1}{\zeta-u}\,  \widetilde{\Gamma}^{(m)}(u;\sigma_1,\ldots,\sigma_m),\end{equation}
\be \lim_{N\rightarrow \infty}  \sqrt{N}\prod_{l=1}^r(\nu_l+N) \, \Upsilon_{N}^{(\mathrm{cri})}  \mathbb{E}\Big[\prod_{l=1}^{N}\big(x_l-\xi/\sqrt{N}\big)\Big]={\Gamma}^{(m+1)}(\xi;\sigma_1,\ldots,\sigma_m)\ee
where \be \Upsilon_{N}^{(\mathrm{cri})}=(-1)^{N}N^{N+(\nu_0-m)/2}e^{\sqrt{N}\tau+\tau^2/2}\prod_{l=1}^r\Gamma(\nu_l+N).\ee

(iii) Set $ a_j= \sigma_j b/(b-1)$ with $\sigma_j>0$ for $j=1, \ldots, m$, if  $b>1$, then    for $r=0$,
\be \lim_{N\rightarrow \infty} \frac{-b}{(b-1)\Upsilon_{N}^{(\mathrm{sup})}}\mathbb{E}\Big[\prod_{l=1}^{N}\frac{1}{x_l-(1-\frac{1}{b})\zeta}\Big]=\int_0^{\infty} du \frac{1}{\zeta-u}\,  \widetilde{\mathcal{L}}^{(m)}(u;1-1/b;\sigma_1,\ldots,\sigma_m),\ee
\be \lim_{N\rightarrow \infty}   \frac{b}{b-1}\prod_{l=1}^r(\nu_l+N) \, \Upsilon_{N}^{(\mathrm{sup})}  \mathbb{E}\Big[\prod_{l=1}^{N}\big(x_l-(1-1/b)\xi\big)\Big]={\mathcal{L}}^{(m+1)}(\xi;1-1/b;\sigma_1,\ldots,\sigma_m)\ee
where \be \Upsilon_{N}^{(\mathrm{sup})}=(-1)^{N}(bN)^{N-m}\big(b/(b-1)\big)^{m}.\ee
\end{theorem}

\begin{proof} By Proposition \ref{averageproduct}, following almost the same   procedure as that  in  Theorems \ref{noncriticalkernel}, \ref{deformedcriticalkernel} and \ref{supercriticalkernel} we can evaluate the scaling limits. As a matter of fact, the proof will be simpler since it only involves a single variable integral. We omit the  details.\end{proof}

Let us conclude this section with two  relationships between  the limiting kernels (cf. \eqref{deformedcriticalk} and \eqref{supercriticalk}) and   the generalised multiple functions defined by \eqref{typeIIpearceylikefunction}--\eqref{typeIlaguerrelikefunction};   cf.~Corollary \ref{pdfquasi}.


\begin{prop} We have
\be {\mathcal K}^{{\rm h},r}_{m}(x,y;\tau,\sigma)={\mathcal K}^{{\rm h},r}(x,y;\tau)+\sum_{k=1}^m \Gamma^{(k)}(x;\sigma_1,\ldots,\sigma_{k-1})\,\widetilde{\Gamma}^{(k)}(y;\sigma_1,\ldots,\sigma_{k}),\ee
and \be {\widetilde{\mathcal K}}^{{\rm h},r}_{m}(x,y;\kappa;\sigma)= \sum_{k=1}^m \mathcal{L}^{(k)}(x;\kappa;\sigma_1,\ldots,\sigma_{k-1})\,\widetilde{\mathcal{L}}^{(k)}(y;\kappa;\sigma_1,\ldots,\sigma_{k}).\ee
\end{prop}

\begin{proof}  By use of the relation \eqref{identicalrelation},  noting the definition of involved functions \eqref{typeIIpearceylikefunction}--\eqref{typeIlaguerrelikefunction},  term-by-term integration immediately implies the above two formulas. Here use has been made of   ${\widetilde{\mathcal K}}^{{\rm h},r}_{0}(x,y;\kappa;\sigma)=0$  for the second formula.   \end{proof}

 \section{Product with truncated unitary matrices} \label{sectiontruncatedunitary}

The derivation  of the double contour integral expression \eqref{kernelCD} for the correlation kernel  is expected to be applicable to a wider class of biorthogonal ensembles, specifically to those  characterized by  the form of \eqref{chGUEsourcepdf} with $\eta_i(x)=x^{i-1}$ and $\xi_i(x)=h(a_i, x)$ for some appropriate   function of two variables $h$ and $N$ generic parameters $a_1, \ldots, a_N$.    In this section we consider the specific case of the biorthogonal ensemble corresponding to the
product of $r$ truncated unitary matrices and one shifted mean Ginibre matrix and derive a double integral representation of the correlation kernel and  analyze the scaled limits at the hard edge. Other types of products $X_r \cdots X_1 Z$, where each $X_j$ is a Ginibre or truncated unitary matrix while $Z$ is a spiked Wishart matrix of the form $G_0 \Sigma $  or a triangular random matrix
(cf.~\cite{Ch14,FW15}), are presently under consideration \cite{liu15}.

Explicitly, instead of (\ref{Y}), we now consider the matrix product
 \begin{equation}\label{Yunitary}
 Y = T_r \cdots T_1 (G_0 + A),
 \end{equation}
where each $T_j$ is an $(N+\nu_j)\times (N+\nu_{j-1})$  truncation of  a Haar distributed    unitary matrix  of size $M_j\times M_j$ and $G_0$ is an $(N+\nu_0)\times N$ standard complex
Gaussian matrix while $A$ is of size $(N + \nu_0) \times N$ and
fixed. Here $\nu_{-1}=0$, $\nu_0, \ldots, \nu_r$ are the nonnegative integers and $\mu_j:=M_j-N>\nu_j$ (for the general $\nu_j>-1$  the analysis below is also  applicable).  In the case that the
matrix $ (G_0 + A)$ is absent, this product has been studied in a recent paper \cite{KKS15}.

An analogue of Proposition \ref{pdf} for the correlation kernel can be given. As in
Proposition \ref{pdf}, two auxiliary functions are required, and so as to stress the structural similarities, analogous notation is used.
Specifically, with $r=1,2,\ldots,$ and $0\leq q \leq r$,  the first is defined to be
 \begin{multline}\Psi_{q}(u;x)= \frac{1}{(2\pi i)^{r}}    {1 \over \Gamma(\nu_0+1)} \int_{(0,\infty)^{q}}dt  \int_{\Gamma}d w \prod_{l=1}^q  t_{l}^{\mu_l}e^{-t_l} \\ \times
\prod_{l=1}^r   w_{l}^{-\nu_l-1}e^{w_l } \exp\Big\{x\frac{t_1\cdots t_q}{w_1\cdots w_r}\Big\}\,
   {}_0F_1\Big(\nu_0+1;- xu\frac{  t_1\cdots t_q}{w_1\cdots w_r}\Big),\label{function4.1}\end{multline}
where $\Gamma=\gamma_1 \times \cdots \times \gamma_r$, and $\gamma_1, \ldots, \gamma_r$ are   paths  starting   and ending at $-\infty$ and  encircling  the origin anticlockwise, while the other reads
  \be \Phi_{q}(v;y)= \frac{1}{2\pi i} \int_{c-i\infty}^{c+i\infty} ds\, y^{-s}  \phi(v;s)  \prod_{l=1}^{r} \Gamma(\nu_l+s) \prod_{l=1}^{q} \frac{1}{\Gamma(\mu_l+s)},\label{function4.2} \ee
 where $\phi(v;s)$ is given in \eqref{function1} and   $c>-\min\{\nu_0, \nu_1, \ldots, \nu_r\}$.  It is worth stressing that $\Psi_{q}(u;x)$ can be expressed as a single contour integral
 \begin{multline}\Psi_{q}(u;x)= {1 \over \Gamma(\nu_0+1)} \frac{1}{ 2\pi i }        \int_{\gamma}d w \,(-x)^{-w} {}_1F_1\big(w;\nu_0+1;u\big) \\ \times \Gamma(w)
 \prod_{l=1}^q  \Gamma(\mu_l+1-w)  \prod_{l=1}^{r} \frac{1}{\Gamma(\nu_l+1-w)}
   ,\label{function4.1-1}\end{multline}
 where $\gamma$    encircles all non-positive integers   such that
 $\text{Re}\{w\}<\min\{\mu_1 +1,\ldots,\mu_q +1\}$ for any $w\in \gamma$.
 This  is   a nice analogue of the definition of the function $\Phi_{q}(v;y)$ and can be derived as follows. First, the power series expansions for  the two functions  $e^{x}$ and  ${}_0F_1\big(\nu_0+1;- xu\big)$ give us the following  relation
 \be e^x {}_0F_1\big(\nu_0+1;- xu\big)=\sum_{k=0}^{\infty}\frac{1}{k!}\, x^k {}_1F_1\big(-k;\nu_0+1;u\big),\ee
from which,  together with the definition of the function \eqref{function4.1}, by term-by-term integration  we then read off
\begin{multline}\Psi_{q}(u;x) =     {1 \over \Gamma(\nu_0+1)} \sum_{k=0}^{\infty}\frac{1}{k!}\, x^k {}_1F_1\big(-k;\nu_0+1;u\big) \\
 \times \prod_{l=1}^q  \Gamma(\mu_l+1+k)  \prod_{l=1}^{r} \frac{1}{\Gamma(\nu_l+1+k)}.  \end{multline}
With this, noting that the integrand for the integral  \eqref{function4.1-1} has simple poles at $0, -1, -2, \ldots$, we thereby apply the residue theorem to  get the desired result.

 \begin{prop}\label{pdfunitary} Let   $Y$ be  defined by   (\ref{Yunitary}), and suppose that all   eigenvalues  $a_1, \ldots, a_N$  of  $A^*A$  are positive.  The eigenvalue PDF of $Y^*Y$ can be
 written as
 \be \mathcal{P}_{N}(x_1,\ldots,x_N)=\frac{1}{N!} \det[K_N(x_i,x_j)]_{i,j=1}^{N}
 \ee
with correlation kernel
 \be
K_N(x,y)=\frac{1}{2\pi i}\int_{0}^{\infty} du \int_{\mathcal{C}} dv \, u^{\nu_0}e^{-u+v} \Psi_r(u;x)\Phi_r(v;y)\frac{1}{u-v}\prod_{l=1}^{N}\frac{u+a_l}{v+a_l}, \label{kernelCDunitary}\ee
where $\mathcal{C}$ is a counterclockwise contour encircling $-a_1,\ldots, -a_N$ but not $u$.
 \end{prop}

\begin{proof} Starting with the eigenvalue PDF \eqref{chGUEsourcepdf} of $(G_0+A)^{*}(G_0+A)$, application of \cite[Corollary 2.4]{KKS15}  $r$ times in succession shows that the eigenvalue  PDF of $Y^*Y$ is proportional to
 \begin{equation}    \det[\eta_i(x_j)]_{i,j=1}^{N}\det[\xi_i(x_j)]_{i,j=1}^{N}, \label{Biunitary}\ee
where $\eta_i(x)=x^{i-1}$ and   with $T=t_1 \cdots t_r$
 \be \label{Tunitary} \xi_i(x)=\frac{1}{\Gamma(\nu_0+1)}\int_{(0,1)^r} d t   \, \prod_{l=1}^{r}  t_{l}^{\nu_l-1} (1-t_{l})^{\mu_l-\nu_l-1} \,  (\frac{y}{T})^{\nu_0} e^{-\frac{y}{T}}   {}_0F_1(\nu_0+1;a_i\frac{x}{T}). \end{equation}

Next, we proceed as in the proof of Proposition \ref{pdf}. Our  first task is to compute $g_{i,j} :=  \int_0^\infty \eta_i(x) \xi_j(x) \, dx$. For this purpose, we note that application of the Mellin transform
gives
  \be \int_{0}^{\infty} y^{s-1}\xi_j(y) \, dy=\phi(-a_j;s) \prod_{l=1}^{r} B(\nu_l+s,\mu_l-\nu_l),\label{mellin1unitary}\ee
where  the notation $B(a,b)$ refers to the gamma function evaluation of the beta integral and $\phi(v;s)$ is given in \eqref{function1},  while use of the inverse Mellin transform gives
  \be  \xi_j(y)=\Phi_r(-a_j;y)\prod_{l=1}^{r} \Gamma(\mu_l-\nu_l), \ee
where $\Phi_r$ is defined in \eqref{function4.2} with $q=r$.
Combining  \eqref{mellin1unitary}, \eqref{function2} and \eqref{laguerreformula1}, we obtain
\be g_{i,j}=(i-1)!e^{a_j}L_{i-1}^{\nu_0}(-a_j)\prod_{l=1}^{r} B(\nu_l+i,\mu_l-\nu_l).\label{gLunitary} \ee

According to Proposition \ref{PB}, with $G=[g_{i,j}]_{i,j=1}^N$ and $C=(G^{-1})^{t}$, the entries $c_{i,j}$ of $C$ then satisfy
\be \label{bvunitary} e^{a_k}\sum_{i=1}^{N}(i-1)!\, L_{i-1}^{\nu_0}(-a_k)\prod_{l=1}^{r} B(\nu_l+i,\mu_l-\nu_l) \, c_{i,j}=\delta_{j,k}.\ee
Without loss of generality we assume that $a_1, \ldots, a_N$ are pairwise distinct. The above equations imply
 \be \sum_{i=1}^{N}(i-1)!\,L_{i-1}^{\nu_0}(u) \prod_{l=1}^{r} B(\nu_l+i,\mu_l-\nu_l)\, c_{i,j}=e^{-a_j}\prod_{l=1,l\neq j}^{N} \frac{-u-a_l}{a_j-a_l}, \label{csumidentityunitary}\ee
which can be verified by noting that both sides are polynomials of degree $N-1$ in $u$ which are equal at $N$ different points. Using this implicit formula for $\{c_{i,j}\}$ and    the  integral representations
  \begin{equation}\label{intrepgamma1unitary}
  \Gamma(z)  =   \int_{0}^{\infty} t^{z-1}e^{-t} dt,  \qquad \frac{1}{\Gamma(z)} = \frac{1}{2\pi i} \int_{\gamma} w^{-z}e^{w} dw,
  \end{equation}
    we have from  \eqref{7.eb2'}  that with $T=t_1\cdots t_r$ and $W=w_1\cdots w_r$
  \begin{align}&K_N(x,y) = \frac{1}{(2\pi i)^{r}}\sum_{i,j=1}^{N}   \xi_j(y)\prod_{l=1}^{r} \frac{1}{\Gamma(\mu_l-\nu_l)}  \int_{(0,\infty)^{r}}dt  \int_{\Gamma}d w   \big(xT/W\big)^{i-1}\times \nonumber \\
  & \, \prod_{l=1}^r \big( t_{l}^{\mu_l}    w_{l}^{-\nu_l-1}e^{w_l-t_l }\big)  \prod_{l=1}^{r} B(\nu_l+i,\mu_l-\nu_l) \, c_{i,j} \nonumber \\
  &= \frac{1}{(2\pi i)^{r}}\sum_{j=1}^{N}   \Phi_r(-a_j;y)  \int_{(0,\infty)^{r}}dt  \int_{\Gamma}d w
 \prod_{l=1}^r \big(t_{l}^{\mu_l} w_{l}^{-\nu_l-1}e^{w_l-t_l }\big)     \frac{e^{ xT/W}}{\Gamma(\nu_0+1)} \,\times
   \nonumber\\
  & \,  \sum_{i=1}^{N} (i-1)! \prod_{l=1}^{r} B(\nu_l+i,\mu_l-\nu_l)\, c_{i,j}
 \int_{0}^{\infty} du \, u^{\nu_0}L_{i-1}^{\nu_0}(u)e^{-u}{}_0F_1\big(\nu_0+1;- uxT/W\big) \nonumber  \\
  &= \frac{1}{(2\pi i)^{r}}\sum_{j=1}^{N}   \Phi_r(-a_j;y)  \int_{(0,\infty)^{r}}dt  \int_{\Gamma}d w
 \prod_{l=1}^r \big(t_{l}^{\mu_l} w_{l}^{-\nu_l-1}e^{w_l-t_l }\big)     \frac{e^{ xT/W}}{\Gamma(\nu_0+1)} \,\times
   \nonumber\\
  &   \,
 \int_{0}^{\infty} du \, u^{\nu_0} e^{-u}{}_0F_1\big(\nu_0+1;-uxT/W\big)\,e^{-a_j}\prod_{l\neq j} \frac{-u-a_l}{a_j-a_l}. \label{2.23unitary}
  \end{align}
  Here the formulae \eqref{laguerreformula2} and \eqref{csumidentityunitary} have been used in the second and third equalities respectively.

  Finally, recalling   \eqref{function4.1}  we  can  rewrite (\ref{2.23unitary}) as
\be
  K_N(x,y) = \int_0^\infty du \, u^{\nu_0} e^{-u} \Psi_r(u;x) \sum_{j=1}^N \Phi_r(-a_j;y)\, e^{-a_j} \prod_{l\ne j} {- u - a_j \over a_j - a_l}.
\ee
 If we recognise the sum over $j$ as the sum of the residues at $\{a_l\}$ of the $v$- function
\be \Phi_r(-v;y)\frac{1}{-u-v}\prod_{l=1}^{N}\frac{-u-a_l}{v-a_l}, \ee
by changing $v$ to $-v$  we then arrive at the desired result.
\end{proof}

At this stage it would be possible to develop the theory of the corresponding averaged characteristic polynomials and their reciprocals,
and then proceed to analyse their hard edge limit; recall Sections \ref{sectionpolynomial} and \ref{limitpolynomial}.
However we pass on this, and instead
analyse the hard edge phase transition analogous to the workings in Section \ref{limitkernel}. Specifically,
taking $N\rightarrow \infty$, we keep all $\nu_j$ fixed and  simultaneously let some of  $\mu_1, \ldots, \mu_r$ go to $\infty$. Without loss of generality, we suppose that for some  $0\leq q\leq r$  all $\nu_1, \ldots, \nu_r, \mu_1, \ldots, \mu_q$ are  constants, and  moreover
\be    \mu_{q+1}, \ldots, \mu_r \rightarrow \infty  \quad \mbox{as}\quad  N\rightarrow \infty;\label{gotoinfinity} \ee
see \cite[Theorem 2.8]{KKS15} for the assumptions.

\begin{theorem}[Subcritical kernel]\label{noncriticalkernelunitary} With the   kernel  \eqref{kernelCDunitary},   for a fixed nonnegative  integer $m$  let
\be a_j=N \sigma_j, \,  j=1, \ldots, m \  \mbox{and} \  a_k=b N, \, k=m+1, \ldots, N,
 \label{confluentsourceunitary}\ee
 where   $0<b<1$ and $\sigma_1, \ldots, \sigma_m>0$. Set $c_N=(1-b)N\mu_{q+1}\cdots \mu_r$.   Under the assumption \eqref{gotoinfinity} we have
\begin{multline}
 \lim_{N\rightarrow \infty}\frac{1}{c_N}K_N\Big(\frac{\xi}{c_N},\frac{\eta}{c_N}\Big)= \\
\int_0^1   G^{1,q}_{q,r+2} \Big ({   -\mu_1,\dots,-\mu_q
 \atop 0,-\nu_0, \dots,-\nu_r} \Big | u\xi \Big ) G^{r+1,0}_{q,r+2} \Big ({ \mu_1,  \dots,\mu_q
 \atop \nu_0,  \dots,\nu_r,0} \Big | u \eta \Big ) \, du. \label{equationsubunitary}
 \end{multline}
  \end{theorem}

 \begin{proof}   Substituting $u, v$  by  $uN, vN$ respectively  in \eqref{kernelCDunitary}, we obtain
 \begin{multline}   \frac{1}{c_N}K_N\Big(\frac{\xi}{c_N},\frac{\eta}{c_N}\Big)=\int_{0}^{\infty}du \int_{\mathcal{C}} dv \,\frac{N}{2c_N\pi i}
     \frac{e^{-N(f(u)-f(v))}}{u-v}
  \\  \qquad \times  \prod_{j=1}^{m}\frac{(v+b)(u+\sigma_j)}{(u+b)(v+\sigma_j)} \,
   (Nu)^{\nu_0} \Psi_r\big(Nu; \frac{\xi}{c_N}\big) \Phi_r\big(Nv; \frac{\eta}{c_N}\big),
    \label{rescalingkernelunitary}
    \end{multline}
     where
    $ f(z)=z-\log(b+z)$.

We can complete the proof  in much the  same way  as  in that of  Theorem \ref{noncriticalkernel}. But this time  we need to estimate the large $N$ leading terms of  the functions $\Psi$ and $\Phi$. That is, we  have to rescale variables $t_{q+1}, \ldots, t_r$  and rewrite them according to
 \begin{multline}\Psi_r\big(Nu; \frac{\xi}{c_N}\big)= \frac{1}{(2\pi i)^{r}}    {1 \over \Gamma(\nu_0+1)}
  \int_{(0,\infty)^{r}}dt  \int_{\Gamma}d w   \, \prod_{j=q+1}^{r} \big( \mu_{j}^{\mu_j+1} e^{ \mu_j(\log t_j-t_j)}\big)\, \times
 \\    \prod_{l=1}^q  t_{l}^{\mu_l}e^{-t_l}
    \prod_{l=1}^r   w_{l}^{-\nu_l-1}e^{w_l }\, \exp\Big\{\frac{\xi}{(1-b)N}\frac{t_1\cdots t_q}{w_1\cdots w_r}\Big\}\,
   {}_0F_1\Big(\nu_0+1;- \frac{\xi u}{1-b}\frac{  t_1\cdots t_q}{w_1\cdots w_r}\Big), \end{multline}
 and
  \begin{multline} \Phi_{r}(Nv; \frac{\eta}{c_N}) =   \frac{1}{2\pi i} \int_{c-i\infty}^{c+i\infty} ds\, \big(\frac{\eta}{(1-b)N}\big)^{-s}  \phi(Nv;s) \\  \times  \prod_{l=q+1}^{r} \frac{\mu_j}{\Gamma(\mu_j+s)}   \prod_{l=1}^{r} \Gamma(\nu_l+s) \prod_{l=1}^{q} \frac{1}{\Gamma(\mu_l+s)} ,  \end{multline}
 then apply the saddle point analysis (see e.g.  \cite{wong01}) to   the integrals over $t_{q+1}, \ldots, t_r$ in $\Psi_r$ near the saddle point  $t_0=1$, or expand the integrand in $\Phi_r$  by the Stirling approximation formula  as $\mu_{q+1}, \ldots, \mu_r \rightarrow \infty$. Tracking the same contour deformations  and following almost the same analysis as in Theorem \ref{noncriticalkernel}, the proof will be done. We leave the details to the reader.
   \end{proof}

The limiting kernel on the RHS of  \eqref{equationsubunitary}, with the parameter $\nu_0$ absent and $r+1$ replaced by $r$
 first appeared in \cite[Theorem 2.8]{KKS15} as the hard edge correlation kernel for a product of truncated unitary matrices. Clearly, it reduces to the Meijer G-kernel \eqref{r}   in case $q=0$. More generally, as remarked in \cite{KKS15} (cf.~eqns (2.37) and (2.38) therein), it can be interpreted as a finite rank perturbation of \eqref{r}.

 For the critical regime, tracking the same contour deformations  and following almost the same analysis as in Theorem  \ref{deformedcriticalkernel}, as for the proof of Theorem \ref{noncriticalkernelunitary}   the required working to establish the following   theorem  can be given.

\begin{theorem} [Deformed critical kernel]\label{deformedcriticalkernelunitary}
 With the kernel    \eqref{kernelCDunitary}, for a fixed nonnegative  integer $m$  let
\be a_j=\sqrt{N}\sigma_j, \,  j=1, \ldots, m \  \mbox{and} \  a_k=N(1- \tau/\sqrt{N})^{-1}, \, k=m+1, \ldots, N,
 \label{mconfluentsourceunitary}\ee
 where $\tau\in \mathbb{R}$ and $\sigma_1, \ldots, \sigma_m>0$.  Set $c_N=\sqrt{N}\mu_{q+1}\cdots \mu_r$. Under the assumption \eqref{gotoinfinity}     we have
\begin{multline}  \lim_{N\rightarrow \infty}\frac{1}{ c_N} K_N\Big(\frac{\xi}{c_N},\frac{\eta}{c_N}\Big)   =   \frac{1}{2\pi i}\int_{0}^{\infty}du \int_{ -c-i\infty}^{-c+i\infty} dv \ \Big (\frac{u}{v} \Big )^{\nu_0}\frac{e^{- \tau u-\frac{1}{2}u^2 +\tau v+\frac{1}{2}v^2}}{u-v}   \\   \times \, \prod_{j=1}^{m}\frac{u+\sigma_j}{v+\sigma_j}\,
 G^{1,q}_{q,r+2} \Big ({   -\mu_1,\dots,-\mu_q
 \atop 0,-\nu_0, \dots,-\nu_r} \Big | u\xi \Big ) G^{r+1,0}_{q,r+2} \Big ({ \mu_1,  \dots,\mu_q
 \atop \nu_0,  \dots,\nu_r,0} \Big | v \eta \Big ),
 \label{deformedcriticalkunitary}\end{multline}
where $0<c<\min\{\sigma_1,\ldots,\sigma_m\}$.
 \end{theorem}

We remark that  the kernels  on the RHS of \eqref{deformedcriticalkunitary} reduce  to the deformed critical kernels ${\mathcal K}^{{\rm h},r}_{m}$ in \eqref{deformedcriticalk}   in case $q=0$. These are  the most general form of critical kernels   that we have  derived in the present paper. Moreover, they are new except for the simplest case $q=r=m=0$, which as previously remarked corresponds to non-intersecting squared Bessel paths and has been studied in  \cite{DV14,DF08,KFW11}.

As to the supercritical regime where $b>1$,  when $r\geq 1$ we have a similar expectation  on   the scaling limit (see   eqn.\eqref{supercriticalk} and relevant description below it), which can be stated as follows.
 With    the kernel    \eqref{kernelCDunitary}, for a fixed positive integer $m$  let
\be a_j=\sigma_j b/(b-1), \,  j=1, \ldots, m \  \mbox{and} \  a_k=bN, \, k=m+1, \ldots, N,
 \label{mconfluentsource2unitary}\ee
 where $b>1$ and $\sigma_1, \ldots, \sigma_m>0$. Set $c_N=  \mu_{q+1}\cdots \mu_r b/(b-1)$, then under the assumption \eqref{gotoinfinity}  we have
\begin{align}  \lim_{N\rightarrow \infty} \frac{1}{c_N} K_N\Big( \frac{\xi}{c_N},  \frac{\eta}{c_N}\Big)  =   & \frac{1}{2\pi i}\int_{0}^{\infty}du \int_{ \mathcal{C}}  dv \  \big(\frac{u}{\kappa}\big)^{\nu_0}\frac{e^{-   u + v}}{u-v} \nonumber \\ & \times \, \,  \Psi_q(u/\kappa;\kappa\xi)\,\Phi_q(v/\kappa;\kappa\eta) \prod_{j=1}^{m}\frac{u+\sigma_j}{v+\sigma_j}\,
 , \label{supercriticalkunitary}\end{align}
where $\kappa=(b-1)/b$, $\Psi_q, \Phi_q$ are given  by \eqref{function4.1}, \eqref{function4.2}, and   $\mathcal{C}$ is a counterclockwise contour encircling  $-\sigma_1, \ldots, -\sigma_m$  but not $u$.

 \section{Asymptotics for large parameters and variables} \label{sectiondiscussion}

 \subsection{Limits for large parameters}

The behavior of   the critical kernel  (\ref{criticalk}) for large values   of the  parameters will be discussed, one of which is the confluent relation  between correlation kernels.
The first to be considered is when some of $\nu_1,\ldots,\nu_r$, say  $\nu_{m+1},\ldots,\nu_r$, go to infinity.

\begin{prop} Let ${\mathcal K}^{{\rm h},r}(\xi,\eta;\tau)$ be the critical kernel \eqref{criticalk}. If $0\leq m<r$, then  as $\nu_{m+1},\ldots,\nu_r \rightarrow \infty$ we have  

\be  (\nu_{m+1} \cdots \nu_r ){\mathcal K}^{{\rm h},r}\big((\nu_{m+1} \cdots \nu_r) x, (\nu_{m+1} \cdots \nu_r) y;\tau\big)\longrightarrow {\mathcal K}^{{\rm h},m}(x,y;\tau).\ee \end{prop}

\begin{proof} This immediately follows from the identity \eqref{fgrelation} for $G_{0,r+2}^{1,0}$ and the definition \eqref{Gfunction} for $G_{0,r+2}^{r+1,0}$. \end{proof}

The above confluent  relation allows  for a natural interpretation,  particularly in the original finite matrix dimension.  Actually, in \eqref{Y}  substituting all $G_j$ as square matrices being distributed according to  the joint density proportional  to
 $\det^{\nu_j}(G_j^* G_j)\exp\{-\textrm{tr}(G_j^* G_j)\}$ (see the  relevant description below \eqref{Y}), then a saddle point approximation shows that all $G_{m+1}, \ldots, G_r$ go to the identity matrix of order $N$ as $\nu_m, \ldots, \nu_r \rightarrow \infty$. Thus these matrices do not contribute to the hard edge state.

A similar effect happens in relation to the parameter $\nu_0$ associated with $G_0$, although now we find that a different rescaling is necessary,
and   furthermore that the limiting kernel is now subcritical.

\begin{prop} \label{largenu0} Let ${\mathcal K}^{{\rm h},r}(\xi,\eta;\tau)$ be the critical kernel \eqref{criticalk}. For $ r\geq 1$,   we have
\be \lim_{ \nu_0 \rightarrow \infty}   \sqrt{\nu_0}  {\mathcal K}^{{\rm h},r}\big(\sqrt{\nu_0} x, \sqrt{\nu_0} y;\tau\big) = {K}^{{\rm h},r-1}(x,y) \Big |_{\{\nu_0,\ldots,\nu_{r-1}\}\rightarrow \{\nu_1,\ldots,\nu_{r}\}},\ee
where $K^{{\rm h}, r-1}$ is given by (\ref{r}).\end{prop}

\begin{proof} Substituting $u, v$ by $ \sqrt{\nu_0}  u$ and $\sqrt{\nu_0}  v$ respectively in \eqref{criticalk}, we get
 \begin{align}  \sqrt{\nu_0}  {\mathcal K}^{{\rm h},r} \big(&\sqrt{\nu_0} x, \sqrt{\nu_0} y ;\tau\big)   =   \frac{1}{2\pi i}\int_{0}^{\infty}du \int_{ i\mathbb{R}} dv \ \frac{e^{ -\nu_0(f(u)-f(v))} }{u-v} e^{- \tau  \sqrt{\nu_0} u + \tau  \sqrt{\nu_0} v}\nonumber \\  \times \, \nu_0
&G_{0,r+2}^{1,0}  \Big({\atop 0, -\nu_0, -\nu_1,\ldots,-\nu_r}\Big|\nu_0 u x\Big) G_{0,r+2}^{r+1,0} \Big({\atop \nu_0,  \nu_1,\ldots,\nu_r,0}\Big|\nu_0 vy\Big)
 \end{align}
 where $f(z)=  -\log z + z^2/2$.

 Choose one saddle point $z_0=1$ from $f'(z)=0$ and deform $i\mathbb{R}$ as the union of one  closed clockwise contour $\mathcal{C}$ encircling the interval $[0,1)$ and the vertical line $x=1$. Note that as $\nu_0 \rightarrow \infty$
 \begin{multline} \nu_0
 G_{0,r+2}^{1,0}  \Big({\atop 0, -\nu_0, -\nu_1,\ldots,-\nu_r}\Big|\nu_0 u x\Big) G_{0,r+2}^{r+1,0} \Big({\atop \nu_0,  \nu_1,\ldots,\nu_r,0}\Big|\nu_0 vy\Big)\\ \sim
 G_{0,r+1}^{1,0}  \Big({\atop 0,   -\nu_1,\ldots,-\nu_r}\Big|  u x\Big) G_{0,r+1}^{r,0} \Big({\atop  \nu_1,\ldots,\nu_r,0}\Big|  vy\Big),\end{multline}
proceeding as in the proof  of Theorem  \ref{noncriticalkernel}, we can show that  the dominant
    contribution    comes   from the range of    $u \in [0,1)$ and $v \in \mathcal{C}$. Finally, application of the residue theorem gives the proof.   \end{proof}

Similarly, for the large negative $\tau$, we observe a transition from the critical kernel to the Meijer G-kernel.
This is to be expected, as then the parameter $b$ in (\ref{AN}) enters the subcritical regime $b < 1$,
since effectively $b = (1 - \tau/N)^{-1}$.

\begin{prop} Let ${\mathcal K}^{{\rm h},r}(\xi,\eta;\tau)$ be the critical kernel \eqref{criticalk}. Then  we have

\be \lim_{ \tau\rightarrow -\infty}    (-1/\tau)  {\mathcal K}^{{\rm h},r}\big(-x/\tau, -y/\tau;\tau\big) = {K}^{{\rm h},r}(x,y).\ee
\end{prop}

\begin{proof} Substituting $u, v$ by $ -\tau  u$ and $-\tau v$ respectively in \eqref{criticalk}, we get
 \begin{align}  (-1/\tau)  {\mathcal K}^{{\rm h},r}&\big(-x/\tau, -y/\tau;\tau\big)  =   \frac{1}{2\pi i}\int_{0}^{\infty}du \int_{ i\mathbb{R}} dv \ \Big (\frac{u}{v} \Big )^{\nu_0} \frac{e^{ -\tau^{2}(f(u)-f(v))} }{u-v}  \nonumber \\  \times
&G_{0,r+2}^{1,0}  \Big({\atop 0, -\nu_0, -\nu_1,\ldots,-\nu_r}\Big|   u x\Big) G_{0,r+2}^{r+1,0} \Big({\atop \nu_0,  \nu_1,\ldots,\nu_r,0}\Big| vy\Big)
 \end{align}
 where $f(z)=  -  z + z^2/2$. Proceeding as in the proof of Proposition  \ref{largenu0}, the sought result follows.
    \end{proof}

Lastly, as to  the critical kernel on the RHS of \eqref{deformedcriticalkunitary} with $m=0$,  the functions  $\widetilde{\Gamma}^{(1)}(x)$ and $\widetilde{\Gamma}^{(0)}(x)$ defined in \eqref{typeIIpearceylikefunction} and \eqref{typeIpearceylikefunction},  there exists  similar asymptotic behavior  for large parameters  as in the above three propositions, but we refrain
 from writing them down.

 \subsection{Conjectures and open problems}
 In the concluding section of \cite{Fo14} a number of questions, mostly relating to asymptotics, were posed in relation to the kernel (\ref{r}). As we will
 indicate, these all carry
 over to the critical kernel (\ref{criticalk}). It is also the case that the conjectured behaviours are all closely related to analogous expected asymptotic properties of
 the finite $N$ kernel (\ref{kernelCD}). Two classes of asymptotic problems stand out.

 The first is to establish the global scaling limit of the critical one-point function. For this we expect
 \begin{equation}\label{E1}
 \lim_{N \to \infty} N^{r}K_N(N^{r+1}x,N^{r+1}x) \Big |_{a_l = N} = {1 \over \pi} {\rm Im} \, G(x - i0),
 \end{equation}
 where $w(z) := z G(z)$, satisfies the algebraic equation
 \begin{equation}\label{AE}
 w^{r+3/2} - z w^{1/2} + z = 0.
 \end{equation}
 The latter is known to specify the Raney distribution with parameters $(3+2r,2)$, which according to free probability theory is the global density
 for the matrix (\ref{Y}) in the critical case (see e.g.~\cite[Remark 3.4]{FL14}).  In the case of the global limit (\ref{E1}) with $a_l = 0$  ($l=1,\dots,N$),
 a recent achievement \cite{LWZ14} has been the use of the double contour integral formula (\ref{bq}) to deduce that (\ref{E1}) with
 $w(z) := z G(z)$ satisfies the algebraic equation
  \begin{equation}\label{AE}
 w^{r+2} - z w + z = 0.
 \end{equation}
 The latter specifies the Raney distribution with parameters $(r+2,1)$, also known as the Fuss-Catalan distribution with parameter $r+1$
 \cite{pz}, and should give the  asymptotic  behavior  of  global density for small argument throughout the subcritical regime. In the supercritical regime, from a macroscopic viewpoint
 the number of random matrices in the product (\ref{Y}) is effectively $r$, since $A$ dominates $G_0$ and moreover $A$ is proportional to
 the identity. This implies that the corresponding asymptotic behaviour of  the global density near the origin  now corresponds to that of the Fuss-Catalan distribution with parameter $r$.

 To see the relevance of (\ref{E1}) to the asymptotics of the density in the critical hard edge scaled state, $\mathcal K^{{\rm h} \, , r}(x,x)$,
 we recall (cf.~\cite[Cor.~2.5]{FL14}) that it can be deduced from (\ref{AE}) that for small $x$ the global density has its leading asymptotics given by
 (\ref{f2}).
  In keeping with the discussion in the concluding section of \cite{Fo14}, this should be the leading large $x$ asymptotic form of
 $\mathcal K^{{\rm h} \, , r}(x,x)$. Combining this with the small $x$ asymptotic form (\ref{f1}) for the Fuss-Catalan density as applies to the
 subcritical and supercritical regimes (the latter with $r \mapsto r -1$ as already commented), we therefore expect
 \begin{equation}\label{R3}
 \mathcal K^{{\rm h} \, , r, \, (*)}(x,x) \mathop{\sim}\limits_{x \to \infty}
 \left \{
 \begin{array}{ll}
 \displaystyle  \frac{1}{\pi}\sin \frac{\pi}{r+2}\,   x^{-1 + \frac{1}{r+2}}, & (*) = {\rm subcritical} \\[.2cm]
  \displaystyle  \frac{1}{\pi}\sin \frac{2\pi}{2r+3}\,   x^{-1 + \frac{1}{r+3/2}}, & (*) = {\rm critical} \\[.2cm]
\displaystyle  \frac{1}{\pi}\sin \frac{\pi}{r+1}\,   x^{-1 + \frac{1}{r+1}}, & (*) = {\rm supercritical}.  \end{array}  \right.
\end{equation}

In general if the global density at the hard edge diverges as $x^{-p}$, then the expected number of eigenvalues in the interval
$(0,s)$ is proportional to $N s^{1-p}$. For this to be of order unity we must scale $s \mapsto N^{1/(1-p)} s$. Taking into consideration the
scaling $x \mapsto N^{r+1} x$ already present in (\ref{E1}), this suggests that the appropriate hard edge scalings are
\begin{equation}
x \mapsto  \left \{
 \begin{array}{ll} (1/N) x, & {\rm subcritical} \\[.2cm]
(1/\sqrt{N})x, & {\rm critical} \\[.2cm]
{\rm no \: change}, &  {\rm supercritical}, \end{array}  \right.
\end{equation}
in agreement with those used in the main body of the text.

 The second class of asymptotic of the type identified in the concluding section of \cite{Fo14}
  is to compute the leading asymptotic form of
 the off diagonal analogue of the LHS of (\ref{E1}), namely
 \begin{equation}\label{Khat}
 \hat{K}_N(x,y) \mathop{:=}\limits^\cdot N^{r+1} K_N(N^{r+1}x,N^{r+1}y), \qquad x \ne y,
 \end{equation}
 where the dot above $:=$ indicates that terms which oscillate and average to zero are to be ignored.
 To see the interest in this quantity, note from (\ref{7.eb2})
 that the truncated (or connected) two-point correlation $\rho_{(2),N}^T(x_1,x_2) := \rho_{(2),N}(x_1,x_2) - \rho_{(1),N}(x_1) \rho_{(1),N}(x_2)$ is given by
 $\rho_{(2),N}^T(x_1,x_2)  = - K_N(x_1,x_2) K_N(x_2,x_1)$, so knowledge of the asymptotics of $\hat{K}(x,y) $ tells us the asymptotics of
 \begin{equation}\label{Wide}
\hat{\rho}_{(2),N}^T(x,y) \mathop{:=}\limits^\cdot N^{2(r+1)} \rho_{(2),N}^T(N^{r+1}x,N^{r+1}y), \qquad x \ne y,
 \end{equation}
 With $G = \sum_{j=1}^N g(x_j)$ denoting a linear statistic in the bulk scaled system, in view of the formula
 (see e.g.~\cite[eqn.~(14.38)]{Fo10})
  \begin{multline}\label{E2}
  {\rm Var} \, G = N^{2(r+1)} \int_0^\infty d x_1 \int_0^\infty dx_2 \, g(x_1) g(x_2) {\rho}_{(2),N}^T(N^{r+1}x_1,N^{r+1}x_2) \\+ N^{r+1} \int_0^\infty g(x) \rho_{(1)}(N^{r+1}x) \, dx
  \end{multline}
  one sees that (\ref{Wide}) (sometimes referred to as a wide correlator; see
  e.g.~\cite{It97})
  essentially determines the large $N$ form of
  this fluctuation, which
  is expected to be $\mathcal{O}(1)$ (see e.g.~\cite[\S 14.3]{Fo10}).

  As a concrete example of this second type of asymptotics, consider the simplest case of (\ref{Y}), namely $r=0$ and $A = 0$. The squared singular values correspond to the eigenvalues of $G_0^* G_0$, where $G_0$ is a $(N + \nu_0) \times N$ standard complex Gaussian
  matrix. This class of random matrices is referred to as the complex Wishart ensemble (see e.g.~\cite[\S 3.2]{Fo10}). For this ensemble it is a known
  result that \cite{Be93}
  \begin{equation}\label{beb}
  N^2\rho_{(2),N}^T(Nx,Ny) \mathop{\sim}\limits^{\cdot} - {1 \over 2 \pi^2} {1 \over (x - y)^2}
  {(L/2) (x + y) - xy \over (x(L-x)y(L-y))^{1/2}}, \qquad x \ne y,
  \end{equation}
  with $L=4$, and where the dot above the asymptotic sign denotes a restriction to non-oscillatory terms.

  Suppose now that in the definition (\ref{Khat})
  of $\hat{K}_N(x,y)$ we introduce a scale factor $L$ and compute instead the asymptotic form of $(N^{r+1}/L) K_N(N^{r+1}x/L,N^{r+1}y/L)$. For the complex Wishart ensemble
  the RHS of (\ref{beb}) with $L$ a variable results. For general $r$, if the original leading asymptotic form of $\tilde{\rho}_{(2)}^T(x,y)$ was $R(x,y)$, this will now equal
  $(1/L^2) R(x/L,y/L)$. Following \cite{Be93} we expect that
 \begin{equation}\label{beb1}
 \lim_{L \to \infty} {1 \over L^2} R \Big ( {x \over L}, {y \over L} \Big ) \to R^{\rm h}(x,y),
 \end{equation}
 where $R^{\rm h}(x,y)$ is the leading non-oscillatory large $x$, large $y$ asymptotic form of the hard edge scaling of
 $\rho_{(2),N}^T(x,y)$.  In the context of the present setting this corresponds to seeking the large $x$, large $y$ form
 of $\mathcal K^{{\rm h}, r}(x,y)$. In the case of the complex Wishart ensemble, (\ref{beb1}) applied to (\ref{beb}) predicts that
\begin{equation}\label{beb2}
\rho_{(2)}^{{\rm h},T}(x,y)  \mathop{\sim}\limits^\cdot  - {1 \over 4 \pi^2}
{ (x/y)^{1/2} + (y/x)^{1/2} \over (x-y)^2},
\end{equation}
which is in fact a known exact result (see e.g.~\cite[eqn.~(7.75)]{Fo10}). The analogue of (\ref{beb2}) is known for the case $r=1$, $A =0$ of
(\ref{Y}) \cite[eqn.~(5.28)]{Fo14}, but the analogue of (\ref{beb}) is yet to be obtained. As discussed in \cite{Fo14}, knowledge of an asymptotic form
such as (\ref{beb2}) is of interest for the computation of the variance of a scaled linear statistic at the hard edge,
$G_\alpha = \sum_{j=1}^\infty g(x_j/\alpha)$ when $\alpha \to \infty$, which is  given by
\begin{eqnarray}\label{VL}
 \lefteqn{
\lim_{\alpha \to \infty} {\rm Var} \, G_\alpha   := \lim_{\alpha \to \infty} \Big ( \int_0^\infty d \lambda_1   \int_0^\infty d \lambda_2 \,}  \nonumber \\ &&
\times g(\lambda_1/\alpha) g(\lambda_2/\alpha)  \rho_{(2)}^{T, {\rm h}, \, r}(\lambda_1,\lambda_2)+ \int_0^\infty d \lambda \, g(\lambda/\alpha)    \rho_{(1)}^{{\rm h},\, r}(\lambda) \Big ) .
\end{eqnarray}

A number of challenges for future research present themselves from the above discussion.
We conclude this section with a list of a few more.

\begin{itemize}
\item Under the assumption of $a_1= \cdots= a_N=b N$ with $b>0$, verify the sine-kernel  in the bulk and Airy-kernel at the soft edge for \eqref{kernelCD} and \eqref{kernelCDunitary} (see recent monographs  \cite{agz09,De99,Fo10,tao12} for the sine and Airy kernels and  \cite{LWZ14} for recent  progress on  the random matrix products).
\item Under the assumption of $a_{m+1}= \cdots= a_N=b N$ with $b>0$, by tuning the  parameters $a_1, \ldots, a_m$ verify the BBP transition for \eqref{kernelCD} and \eqref{kernelCDunitary} (cf. \cite{BBP,Pe06}).
\item Verify the transitions from the critical kernels \eqref{criticalk} and  \eqref{deformedcriticalkunitary} to the sine-kernel and to the Airy-kernel (cf.~\cite[Exercise 7.2]{Fo10}  and \cite{Fo14}).
\end{itemize}

\begin{acknow}
The work of P.J.~Forrester was supported by the Australian Research Council  for the project DP140102613.  The work of D.-Z.~Liu  was  supported by the National Natural Science Foundation of China under grants  11301499 and  11171005.
 Special thanks go to  Dong Wang for inviting us to the  Department of Mathematics at NUS in  July  2014, and to Lun Zhang for bringing  the preprint of \cite{DV14}  to our attention during the drafting of this article. The anonymous referees' constructive comments and  suggestions  are most appreciated.

\end{acknow}

\bibliographystyle{amsplain}

\end{document}